\documentclass[english]{article}
\usepackage{amsmath,amssymb,latexsym,amsthm,mathrsfs,appendix}
\usepackage{color}
\usepackage{scalerel}
\usepackage{graphicx}
\usepackage{caption}
\usepackage{subcaption}
\usepackage{mathtools}
\usepackage{tikz}
\usepackage{capt-of}

\date{}

\def\theenumi{\arabic{enumi}}

\def\theenumii{\alph{enumii}}
\def\p@enumii{\theenumi.}

\def\theenumiii{\arabic{enumiii}}
\def\p@enumiii{(\theenumi)(\theenumii)}

\def\p@enumiv{\p@enumiii.\theenumiii}
\newtheorem{theorem}{Theorem}[section]

\newtheorem{corollary}[theorem]{Corollary}

\newtheorem{proposition}[theorem]{Proposition}
\newtheorem{lemma}[theorem]{Lemma}

\theoremstyle{definition}
\newtheorem{example}[theorem]{Example}

\newtheorem{definition}[theorem]{Definition}
\newtheorem{remark}[theorem]{Remark}

\newtheorem{assumption}[theorem]{Assumption}

\DeclareMathOperator*{\esssup}{ess\,sup}

\makeatother

\usepackage{babel}
\begin{document}
\title{Rough Hypoellipticity for the Heat Equation in Dirichlet Spaces}
\author{Qi Hou\thanks{%
Partially supported by NSF grant DMS  1404435 and DMS 1707589} \,
and Laurent Saloff-Coste\thanks{ Partially supported by NSF grant DMS  1404435 and DMS 1707589} \\
{\small Department of Mathematics}\\
{\small Cornell University}  }
\maketitle
\begin{abstract}
This paper aims at proving the local boundedness and continuity of solutions of the heat equation in the context of Dirichlet spaces under some rather weak additional assumptions. We consider symmetric local regular Dirichlet forms which satisfy mild assumptions concerning (a) the existence of cut-off functions, (b) a local ultracontractivity hypothesis, and (c) a weak off-diagonal upper bound. In this setting, local weak solutions of the heat equation, and their time derivatives, are shown to be locally bounded; they are further locally continuous, if the semigroup admits a locally continuous density function. Applications of the results are provided including discussion on the existence of locally bounded heat kernel; $L^\infty$\ structure results for ancient solutions of the heat equation. The last section presents a special case where the $L^\infty$\ off-diagonal upper bound follows from the ultracontractivity property of the semigroup. This paper is a continuation of \cite{L2}.
\end{abstract}

\section{Introduction}
\setcounter{equation}{0}
In this paper we study the local boundedness and continuity properties of local weak solutions of heat equations in the setting of symmetric local regular Dirichlet spaces and other more general settings. Let $(X,d,m)$\ be a metric measure space where $X$\ is locally compact separable and $m$\ is a Radon measure with full support, $d$\ is some metric on $X$\ that we omit writing in the rest of the paper. Let $(\mathcal{E}, \mathcal{F})$\ be a symmetric regular local Dirichlet form on $L^2(X,m)$\ with domain $\mathcal{F}$. Let $(H_t)_{t>0}$\ and $-P$\ be the associated semigroup (referred to as the heat semigroup) and generator. We call $(X,m,\mathcal{E},\mathcal{F})$\ a Dirichlet space. The following three items list our main assumptions on the Dirichlet space.
\begin{itemize}
	    \item[(i)] The Dirichlet space admits nice cut-off functions, in the sense that there exists some topological basis $\mathcal{TB}$\ of $X$, that for any $U,V\in\mathcal{TB}$\ with $V\Subset U\Subset X$, for any $0<C_1<1$, there exists a positive constant $C_2(C_1,V,U)$\ and a cut-off function $\eta$\ satisfying $\eta\equiv 1$\ on $V$, $\mbox{supp}\{\eta\}\subset U$, such that for any function $f$\ in $\mathcal{F}$,
	    \begin{eqnarray*}
	    \int_Xf^2\,d\Gamma(\eta,\,\eta)\leq C_1\int_X\eta^2\,d\Gamma(f,\,f)+C_2\int_{\scaleto{\mbox{supp}\{\eta\}}{5pt}}f^2\,dm.
	    \end{eqnarray*}
	    Here $\Gamma$\ denotes the energy measure associated with the Dirichlet form $\mathcal{E}$.
	    
	    Such cut-off functions are generalizations of cut-off functions with bounded gradient. Many fractal spaces, including the Sierpinski gasket and countable products of Sierpinski gaskets, admit such cut-off functions.
	    \item[(ii)] The semigroup $H_t$\ is locally ultracontractive. That is, for any precompact open subset $\Omega\Subset X$, there exists some finite interval $(0,T)$, $T>0$, and some continuous nonincreasing function $M_\Omega(t):(0,T)\rightarrow \mathbb{R}_+$, such that
		\begin{eqnarray*}
		\left|\left|H_t\right|\right|_{L^2(\Omega)\rightarrow L^\infty(\Omega)}\leq e^{M_\Omega(t)}.
		\end{eqnarray*}
		When the semigroup is globally ultracontractive, i.e., for the whole space $X$, there is a function $M(t)$\ satisfying the same properties, such that $||H_t||_{L^2(X)\rightarrow L^\infty(X)}\leq e^{M(t)}$, it is clear that the semigroup admits an essentially bounded density function. We show in this paper that when there exist nice cut-off functions, the local ultracontractivity condition together with the $L^\infty$\ off-diagonal upper bound for the semigroup defined in the next item, implies the existence of a locally bounded density function $h(t,x,y)$.
		\item[(iii)] $H_t$\ further satisfies the $L^\infty$\ off-diagonal upper bound. That is, for any precompact open sets $V,W\Subset X$\ with $\overline{V}\cap \overline{W}=\emptyset$, any $n\in \mathbb{N}$, there exists some positive function $G(V,W,n,t)=:G(n,t)$, continuous in $t$\ ($t\in (0,T)$\ for some $T>0$), such that for any $v,w\in L^1(X)$\ with $\mbox{supp}\{v\}\subset V$, $\mbox{supp}\{w\}\subset W$, 
		\begin{eqnarray*}
		\left|\int_X\partial_t^nH_tv\cdot w\,dm\right|\leq e^{-G(n,t)}\left|\left|v\right|\right|_{L^1}\left|\left|w\right|\right|_{L^1},
		\end{eqnarray*}
		where $G(n,t)$\ satisfies that for any $a\geq 0$, any $n\in \mathbb{N}$,
        \begin{eqnarray*}
        \lim_{t\rightarrow 0^+}\frac{1}{t^a}e^{-G(n,t)}<\infty.
        \end{eqnarray*}
        Local ultracontractivity sometimes implies such $L^\infty$\ off-diagonal upper bound. In the last section we discuss one such case.
	\end{itemize}
Often when the semigroup admits a density function (called the heat kernel) that satisfies some appropriate Gaussian estimates, it follows that the semigroup satisfies (ii) and (iii). These two items can be considered as a relaxation of good Gaussian estimates of heat kernels. One main result we present in this paper is the following theorem.
\begin{theorem}
\label{modelthm}
	Under the above assumptions, given any open subset $U\subset X$\ and any open interval $I=(a,b)\subset \mathbb{R}$, given any function $f$\ that is locally in $W^{n,\infty}(I\rightarrow L^\infty(U))$\ where $n\in \mathbb{N}$, let $u$\ be any local weak solution of $(\partial_t+P)u=f$\ on $I\times U$. Then $u$\ is locally in $W^{n,\infty}(I\rightarrow L^\infty(U))$. Moreover, if the the semigroup $(H_t)_{t>0}$\ admits a density function $h(t,x,y)$\ that is continuous on $(0,c)\times V\times V$\ for some $(0,c)\subset (0,T)$, $V\subset U$, then $u$\ is continuous on $I\times V$.
\end{theorem}
Theorem \ref{modelthm} can be viewed as a continuation of the main result in the preceding paper \cite{L2}, where under the assumption of existence of nice cut-off functions (item (i) above) and the assumption of a very weak $L^2$\ Gaussian type upper bound (corresponding to but much weaker than item (iii)), we showed that any time derivative of any local weak solution of the heat equation is still a local weak solution (in particular, still locally belongs to the function space $L^2(I\rightarrow\mathcal{F})$). In this perspective, the result in \cite{L2} is an $L^2$-type result on time derivatives of local weak solutions; Theorem \ref{modelthm}, on the other hand, addresses the $L^\infty$-type properties of a local weak solution itself, and can be built on the $L^2$\ result to conclude $L^\infty$-type results for time derivatives of local weak solutions.

We also remark that our assumptions on the Dirichlet space have some overlaps with, but are overall distinct from the assumptions made in some papers (e.g. \cite{Sturm3,Lierlfractal,Lierl2}) that discuss the H\"{o}lder continuity property of local weak solutions of heat equations. In particular, we do not assume that the Dirichlet space satisfies the volume doubling property of any type, or the scale-invariant parabolic Harnack inequality. Our approach is via approximation using the heat semigroup, hence our assumptions are mainly on the semigroup. For the same reason, because we make much use of good properties of time derivatives of the heat semigroup, we have to restrict ourselves to time-independent Dirichlet forms.

A second goal of this paper is to illustrate how to draw similar conclusions for certain heat equation solutions in settings with very local assumptions. We present a general framework to conclude such local boundedness and continuity results from, where Theorem \ref{modelthm} is a special case. Essentially, suppose we have a class of functions on $I\times X$, where $I$\ is some open interval in $\mathbb{R}$. Let $\Omega\Subset X$\ be a precompact open subset. If there is a heat equation of which these functions are local weak solutions when restricted to the subset $I\times\Omega$, and if the heat equation is associated with some Dirichlet form on $\Omega$\ that satisfies the assumptions in Theorem \ref{modelthm} (i.e. the form is symmetric local regular and admits nice cut-off functions for any pair of precompact open subsets of $\Omega$; its corresponding semigroup is locally ultracontractive and satisfies some $L^\infty$\ off-diagonal upper bound in $\Omega$), then this class of functions locally belongs to $W^{n,\infty}(I\rightarrow L^\infty(\Omega))$, given that the right-hand side of the heat equation is locally in $W^{n,\infty}(I\rightarrow L^\infty(\Omega))$.

This general framework covers scenarios including heat equations associated with
\begin{itemize}
    \item locally uniformly elliptic divergence form second order differential operators with measurable, locally bounded coefficients on $\mathbb{R}^n$;
    \item certain divergence form operators on domains in $\mathbb{R}^n$\ with curl-free drifts that locally have potentials;
    \item certain heat equations on polyhedral complexes; in particular, we give examples of heat equations on one-dimensional complexes associated with locally finite graphs.
\end{itemize}
The key feature of the latter two examples is that only locally can the heat equation be interpreted as coming from a Dirichlet form for sure. We explain in more details these examples in Section \ref{frameworksection}.

This paper is organized as follows. After a quick review of basic notions and properties of Dirichlet forms that we use in this paper, in Section 2 we define local weak solutions and make precise the assumption on existence of nice cut-off functions. Section 3 contains the statement and proof of a simpler version of Theorem \ref{modelthm}, where we replace the assumption of local ultracontractivity of the heat semigroup by global ultracontractivity, and show that local weak solutions (and their time derivatives) are locally bounded or continuous. Next in Section 4 we formulate the general framework where such results hold and give some examples. It will be evident that Theorem \ref{modelthm} is a special case in this general framework. As applications of the main theorem, we show (1) the implication of the existence of a locally bounded density function from the local ultracontractivity property of the semigroup (in the presence of some technical hypotheses); (2) as a continuation of the application presented at the end of the preceeding paper \cite{L2}, we explain here the $L^\infty$-type structure results for ancient local weak solutions of the heat equation. These applications are the topic of Section 6. Lastly, in the Appendix, we discuss one case where the local ultracontractivity condition implies the $L^\infty$\ off-diagonal upper bound for the semigroup. For discussions of such implications under various assumptions in the literature, cf. e.g. \cite{kernel,anomalous}. In this paper the approach we present is built entirely on iterations of the $L^2$\ off-diagonal upper bound. It covers cases not covered in the existing literature.

\section{Local weak solutions and assumption on cut-off functions}
\setcounter{equation}{0}
Let $(X,m)$\ be a metric measure space where $m$\ is a Radon measure with full support and $X$\ is locally compact separable. In the rest of the paper we simply say $(X,m)$\ is a metric measure space and we do not assign a name for the metric. Let $(\mathcal{E},\mathcal{F})$\ be a Dirichlet form on $L^2(X,m)$; $\mathcal{F}$\ denotes the domain of $\mathcal{E}$. $(\mathcal{E},\mathcal{F})$\ equipped with the $\mathcal{E}_1$\ norm is a Hilbert space, where for any $f\in \mathcal{F}$,
\begin{eqnarray*}
||f||_{\mathcal{E}_1}=\left(\mathcal{E}_1(f,f)\right)^{1/2}=\left(||f||_{L^2(X,m)}^2+\mathcal{E}(f,f)\right)^{1/2}.
\end{eqnarray*}
Let $(H_t)_{t>0}$\ and $-P$\ be the self-adjoint semigroup and (infinitesimal) generator associated with $(\mathcal{E},\mathcal{F})$, respectively. Let $\mathcal{D}(P)$\ denote the domain of the generator $-P$. $\mathcal{D}(P)$\ is dense in $\mathcal{F}$\ w.r.t. the $\mathcal{E}_1$\ norm. For any $t>0$, $H_t$\ is a contraction on $L^2(X)$, and for any $k\in \mathbb{N}_+$,
\begin{eqnarray*}
||P^kH_t||_{L^2(X)\rightarrow L^2(X)}\leq (k/et)^k
\end{eqnarray*}
by the spectral theory. For any $u_0\in L^2(X)$, $H_tu_0$\ is smooth in $t$, and satisfies
\begin{eqnarray*}
\partial_tH_tu_0=-PH_tu_0
\end{eqnarray*}
in $L^2(X)$. In particular, $u(t,x):=H_tu_0(X)$\ solves the PDE
\begin{eqnarray*}
(\partial_t+P)u=0
\end{eqnarray*}
in the strong sense.

Recall the following properties that a Dirichlet form may or may not satisfy
\begin{itemize}
    \item $(\mathcal{E},\mathcal{F})$\ is called {\em local}, if for any $f,g\in\mathcal{F}$, $\mathcal{E}(f,g)=0$\ whenever $\mbox{supp}\{f\}$, $\mbox{supp}\{g\}$\ are disjoint and compact. 
    \item $(\mathcal{E},\mathcal{F})$\ is called {\em regular}, if $\mathcal{C}_c(X)\cap\mathcal{F}$\ is dense in $\mathcal{C}_c(X)$\ in the sup norm and dense in $\mathcal{F}$\ in the $\mathcal{E}_1$\ norm. Here $\mathcal{C}_c(X)$\ represents the set of continuous functions with compact support in $X$.
    \item Any symmetric regular local Dirichlet form admits the Beurling-Deny decomposition
    \begin{eqnarray*}
    \mathcal{E}(f,g)=\int_Xd\Gamma(f,g)+\int_Xfg\,dk,
    \end{eqnarray*}
    where $\Gamma$\ is called the energy measure and $k$\ the Killing measure. When there is no killing measure, the Dirichlet form is called {\em strongly local}, which can be equivalently characterized as satisfying that $\mathcal{E}(f,g)=0$\ for any $f,g\in\mathcal{F}$\ where one function equals to a constant (quasi-a.e.) on the support of the other function.
\end{itemize}
Finally we recall that the energy measure satisfies the following properties. We do not specify quasi-continuous modifications of functions.
\begin{itemize}
    \item (Leibniz rule) For any $u,v,w\in\mathcal{F}$\ with $uv\in\mathcal{F}$\ (e.g. when $u,v\in \mathcal{F}\cap L^\infty$),
    \begin{eqnarray*}
    d\Gamma(uv,w)=u\,d\Gamma(v,w)+v\,d\Gamma(u,w);
    \end{eqnarray*}
    \item (chain rule) For any $u,v\in\mathcal{F}$, any $\Phi\in C^1(\mathbb{R})$\ with bounded derivative and satisfies $\Phi(0)=0$,
    \begin{eqnarray*}
    d\Gamma(\Phi(u),v)=\Phi'(u)\,d\Gamma(u,v);
    \end{eqnarray*}
    \item (Cauchy-Schwartz inequality) For any $f,g,u,v\in \mathcal{F}\cap L^\infty$\ (more generally, when $u,v\in \mathcal{F}\cap L^\infty$\ and $f\in L^2(X,\Gamma(u,u))$, $g\in L^2(X,\Gamma(v,v))$),
    \begin{eqnarray*}
    \left|\int_X fg\,d\Gamma(u,v)\right|&\leq& \left(\int_Xf^2\,d\Gamma(u,u)\right)^{1/2}\left(\int_Xg^2\,d\Gamma(v,v)\right)^{1/2}\\
    &\leq& \frac{C}{2}\int_Xf^2\,d\Gamma(u,u)+\frac{1}{2C}\int_Xg^2\,d\Gamma(v,v).
    \end{eqnarray*}
The last inequality holds for any $C>0$. The corresponding inequality for measures is
\begin{eqnarray*}
|fg|\,d\left|\Gamma(u,v)\right|\leq \frac{C}{2}f^2\,d\Gamma(u,u)+\frac{1}{2C}g^2\,d\Gamma(v,v).
\end{eqnarray*}
    \item (strong locality) For any $u,v\in \mathcal{F}$, if on some open subset $U\Subset X$, $v\equiv C$\ for some constant $C$, then
    \begin{eqnarray*}
    1_U\,d\Gamma(u,v)=0.
    \end{eqnarray*}
\end{itemize}
See \cite{Fukushima} for detailed discussions on symmetric Dirichlet forms.

Let $(\mathcal{E},\mathcal{F})$\ be a symmetric regular local Dirichlet form. To define the notion of local weak solutions, we first define some function spaces associated with $(\mathcal{E},\mathcal{F})$. For any open subset $U\subset X$, any open interval $I\subset \mathbb{R}$, define
\begin{eqnarray*}
&&\mathcal{F}_c(U):=\left\{f\in\mathcal{F}:\mbox{supp}\{f\}\subset U\right\};\\[0.05in]
&&\mathcal{F}_{\scaleto{\mbox{loc}}{5pt}}(U):=\left\{f\in L^2_{\scaleto{\mbox{loc}}{5pt}}(U):\forall V\Subset U\ \exists f^\sharp\in\mathcal{F}\ \mbox{s.t. }f^\sharp=f\ m-\mbox{a.e. on }V\right\};\\
&&\mathcal{F}(I\times X):=L^2(I\rightarrow\mathcal{F})=\left\{f:||f||_{\mathcal{F}(I\times X)}=\left(\int_I\mathcal{E}_1(f,f)\,dt\right)^{1/2}<\infty\right\};\\
&&\mathcal{F}_c(I\times U):=\left\{f\in\mathcal{F}(I\times X):\mbox{supp}\{f\}\subset I\times U\right\};\\[0.05in]
&&\mathcal{F}_{\scaleto{\mbox{loc}}{5pt}}(I\times U):=\\
&&\left\{f:\forall J\times V\Subset I\times U\ \exists f^\sharp\in\mathcal{F}(I\times X)\ \mbox{s.t. }f^\sharp=f\ \mbox{a.e. on }J\times V\right\}.
\end{eqnarray*}
Note that $\mathcal{F}_c(I\times X)\subset \mathcal{F}(I\times X)\subset L^2(I\times X)\subset \mathcal{F}_{\scaleto{\mbox{loc}}{5pt}}(I\times X)\subset L^2_{\scaleto{\mbox{loc}}{5pt}}(I\times X)$. The same relation holds with $I\times X$\ replaced by $X$. We also consider function spaces including the space of compactly supported smooth functions from $I$\ to $\mathcal{F}$, $C_c^\infty(I\rightarrow \mathcal{F})$, and the index-$k$\ Sobolev space from $I$\ to $\mathcal{F}$, $W^{k,2}(I\rightarrow\mathcal{F})$. More details on such function spaces can be found in \cite{Wloka}.

Equating $L^2(X,m)=\left(L^2(X,m)\right)'$\ and let $\mathcal{F}'$\ denote the dual space of $\mathcal{F}$\ w.r.t. the $L^2$\ pairing. Then $\mathcal{F}\subset L^2(X)\subset \mathcal{F}'$, and 
\begin{eqnarray*}
\left(\mathcal{F}(I\times X)\right)'=\left(L^2(I\rightarrow\mathcal{F})\right)'=L^2(I\rightarrow\mathcal{F}').
\end{eqnarray*}
The ``$\sim_{\scaleto{\mbox{loc}}{5pt}}$'' spaces defined above can be viewed as on the dual space end.

We are now ready to define the notion of local weak solutions of the heat equation.
\begin{definition}[local weak solution]
Let $(X,m)$\ be a metric measure space and $(\mathcal{E},\mathcal{F})$\ be a symmetric local regular Dirichlet form on $L^2(X,m)$. Let $U\subset X$\ be an open subset, let $f$\ be a function locally in $L^2\left(I\rightarrow \mathcal{F}'\right)$. We say that $u$\ is a {\em local weak solution} of the heat equation $(\partial_t+P)u=f$\ on $I\times U$, if $u\in \mathcal{F}_{\scaleto{\mbox{loc}}{5pt}}\left(I\times U\right)$, and for any $\varphi\in \mathcal{F}_c\left(I\times U\right)\cap C_c^\infty\left(I\rightarrow \mathcal{F}\right)$,
	\begin{eqnarray}
	\label{localweaksol}
	-\int_I\int_X\ u\cdot \partial_t\varphi\,dm dt+\int_I\mathcal{E}(u,\varphi)\,dt\ =\int_I<f,\,\varphi>_{\scaleto{\mathcal{F}',\mathcal{F}}{5pt}}\,dt.
	\end{eqnarray}
\end{definition}
Here $u$\ in the integral is understood as $u^\sharp$\ as in the definition for $\mathcal{F}_{\scaleto{\mbox{loc}}{5pt}}(I\times U)$\ (relative to the support of $\varphi$). We take this convention throughout this paper. Note that $\mathcal{E}(u,\varphi)$\ is well-defined (independent of the choice of $u^\sharp$) by the local property of $\mathcal{E}$. $<\cdot,\cdot>_{\scaleto{\mathcal{F}',\mathcal{F}}{5pt}}$\ stands for the $(\mathcal{F}',\mathcal{F})$\ pairing.
\begin{remark}
Under the assumption on existence of nice cut-off functions below, any local weak solution $u$\ of $(\partial_t+P)u=f$\ is automatically in $W^{1,2}(I\rightarrow \mathcal{F}')$, cf. \cite{L2,Nate}.
\end{remark}
Lastly we state the assumption on existence of cut-off functions that we take throughout the paper. A cut-off function is any function in $\mathcal{F}\cap\mathcal{C}(X)$\ that is in between $0$\ and $1$.
\begin{assumption}[existence of nice cut-off functions]
\label{cutoff}
There exists some topological basis $\mathcal{TB}$\ of $X$\ so that for any $U,V\in\mathcal{TB}$\ with $V\Subset U\Subset X$, for any $0<C_1<1$, there exist a constant $C_2(C_1,V,U)>0$\ and a cut-off function $\eta$\ satisfying that $\eta\equiv 1$\ on $V$\ and $\mbox{supp}\{\eta\}\subset U$, such that for any function $f$\ in $\mathcal{F}$,
	 \begin{eqnarray}
	 \int_Xf^2\,d\Gamma(\eta,\eta)\leq C_1\int_X\eta^2\,d\Gamma(f,f)+C_2\int_{\scaleto{\mbox{supp}\{\eta\}}{5pt}}f^2\,dm.
	 \end{eqnarray}
Such a function $\eta$\ is called a {\em nice cut-off function} for the pair $V\subset U$\ corresponding to $C_1,C_2$.
\end{assumption}
We review a few properties of nice cut-off functions, details are provided in \cite{L2}.
\begin{itemize}
    \item Let $U$\ be any open subset of $X$. For any function $f\in \mathcal{F}_{\scaleto{\mbox{loc}}{5pt}}(U)$, any nice cut-off function $\eta$\ with $\mbox{supp}\{\eta\}\subset U$, the product $\eta f\in \mathcal{F}_c(U)$.
    \item Assumption \ref{cutoff} implies the existence of nice cut-off functions for any pair of open sets $V\Subset U\Subset X$, where $U,V$\ not necessarily belong to the topological basis $\mathcal{TB}$.
    \item Let $\eta$\ be a nice cut-off function corresponding to $C_1,C_2$. Suppose $0<C_1<\frac{1}{8}$. Then for any $f\in \mathcal{F}_{\scaleto{\mbox{loc}}{5pt}}(X)$, the following inequality holds
    \begin{eqnarray}
    \label{gradineq}
    \int_Xd\Gamma(\eta f,\,\eta f)\leq 2\int_Xd\Gamma(\eta^2f,\,f)+2C_2\int_{\scaleto{\mbox{supp}\{\eta\}}{5pt}} f^2\,dm.
    \end{eqnarray}
    We refer to this inequality as the {\em gradient inequality}.
\end{itemize}

\section{Local boundedness and continuity of local weak solutions under global assumptions}
\setcounter{equation}{0}
We discuss in this section the local boundedness and continuity properties of local weak solutions under relatively strong (global) assumptions. The proof of Theorem \ref{mainthm0} serves as the foundation for deducing the local boundedness property from various global or local assumptions discussed in this paper.
\label{globalsection}
\begin{theorem}
	\label{mainthm0}
	Let $(X,m)$\ be a metric measure space and $(\mathcal{E}, \mathcal{F})$\ be a symmetric regular local Dirichlet form satisfying Assumption \ref{cutoff} (existence of nice cut-off functions). Let $(H_t)_{t>0}$\ and $-P$\ be the associated semigroup and generator. Assume that $(H_t)_{t>0}$\ satisfies 
	\begin{itemize}
		\item the global ultracontractivity property: there is some continuous nonincreasing function $M(t):(0,T)\rightarrow \mathbb{R}_+$, $T>0$, such that
		\begin{eqnarray}
		\label{globalultra}
		\left|\left|H_t\right|\right|_{L^2(X)\rightarrow L^\infty(X)}\leq e^{M(t)};
		\end{eqnarray}
		\item the $L^\infty$\ off-diagonal upper bound: for any precompact open sets $V,W\Subset X$\ with $\overline{V}\cap \overline{W}=\emptyset$, for any $n\in \mathbb{N}$, there exists some positive function $G(V,W,n,t)=:G(n,t)$, continuous in $t\in (0,T)$, such that for any $v,w\in L^1(X)$\ with $\mbox{supp}\{v\}\subset V$, $\mbox{supp}\{w\}\subset W$, 
		\begin{eqnarray}
		\label{globalGaussian}
		\left|\int_X\partial_t^nH_tv\cdot w\,dm\right|\leq e^{-G(n,t)}\left|\left|v\right|\right|_{L^1}\left|\left|w\right|\right|_{L^1},
		\end{eqnarray}
		where $G(n,t)$\ satisfies that for any $a\geq 0$, any $n\in \mathbb{N}$,
        \begin{eqnarray}
        \label{mcontrolpoly}
        \lim_{t\rightarrow 0^+}\frac{1}{t^a}e^{-G(n,t)}<\infty.
        \end{eqnarray}
	\end{itemize}
	Given any open subsets $U\subset X$\ and $I=(a,b)\subset \mathbb{R}$, any function $f$\ that is locally in $W^{n,\infty}\left(I\rightarrow L^\infty(U)\right)$\ for some $n\in \mathbb{N}$, let $u$\ be a local weak solution of $(\partial_t+P)u=f$\ on $I\times U$. Then $u$\ is locally in $W^{n,\infty}(I\rightarrow L^\infty(U))$.
\end{theorem}
\begin{remark}
Theorem \ref{mainthm0} captures the main types of prerequisites we need in showing the local boundedness of a local weak solution and of its time derivatives. In later sections we explore the local nature of this theorem by relaxing all global requirements into local ones; we also present a more general framework and provide examples to show the versatility of the theorem. To further simplify the required conditions, in the Appendix we point out one case where the $L^\infty$\ off-diagonal upper bound follows automatically from the ultracontractivity condition.
\end{remark}
\begin{remark}
The condition $\left|\left|H_t\right|\right|_{\scaleto{L^2(X)\rightarrow L^\infty(X)}{6pt}}\leq e^{M(t)}$\ is equivalent to the condition $\left|\left|H_t\right|\right|_{\scaleto{L^1(X)\rightarrow L^\infty(X)}{6pt}}\leq e^{M_1(t)}$\ for some $M_1(t)$. We write $||\cdot||_{2\rightarrow \infty}$\ and $||\cdot||_{1\rightarrow \infty}$\ when there is no ambiguity. More precisely, the condition on $\left|\left|H_t\right|\right|_{2\rightarrow \infty}$\ implies the corresponding condition on $\left|\left|H_t\right|\right|_{1\rightarrow \infty}$\ by the self-adjointness of $H_t$\ and duality $\left|\left|H_t\right|\right|_{2\rightarrow \infty}=\left|\left|H_t\right|\right|_{1\rightarrow 2}$. Thus it suffices to take $M_1(t)\geq 2M(t/2)$. Conversely, assume the condition on $\left|\left|H_t\right|\right|_{1\rightarrow \infty}$, the condition on $\left|\left|H_t\right|\right|_{2\rightarrow \infty}$\ can be deduced by recalling $\left|\left|H_t\right|\right|_{\infty\rightarrow \infty}\leq 1$\ and using interpolation. So in this direction, $M(t)\geq M_1(t)/2$\ suffices. We denote the relation between $M(t)$\ and $M_1(t)$\ as $M(t)\simeq M_1(t)$.
\end{remark}
\begin{proof}[Proof of Theorem \ref{mainthm0}]
The proof of Theorem \ref{mainthm0} is similar in structure to its ``$L^2$\ counterpart'' in \cite{L2} (see Theorem 4.1 there), with some subtle differences. To show $u$\ is locally bounded in $I\times U$, the main strategy is to show that an approximate sequence $\{\overline{\psi}\widetilde{u}_\tau\}_{\tau>0}$, where $\overline{\psi}$\ is an arbitrary nice product cut-off function supported in $I\times U$\ and $\widetilde{u}_\tau$\ is defined by
\begin{eqnarray}
\label{approxseqdef}
\widetilde{u}_\tau(s,x):=\int_I \rho_\tau(s-t)H_{s-t}(\overline{\eta}^tu^t)(x)\,dt,
\end{eqnarray}
is a Cauchy sequence in $L^\infty(I\times X)$. Here $\rho_\tau(t)$, $\overline{\eta}(t,x)$\ are properly chosen cut-off functions, and we use the notation $v^t(x)$\ to represent $v(t,x)$, the function in $x$\ obtained by fixing $t$\ in any function $v(t,x)$. In classical settings, (\ref{approxseqdef}) is a convolution in space and time. Combining with the fact that the approximate sequence converges to $\overline{\psi} u$\ in $L^2(I\times X)$\ by Proposition 5.3 in \cite{L2}, it follows that $u$\ is locally in $L^\infty(I\times U)$. Repeating the proof for $u$\ on functions $\partial_t^ku$\ as local weak solutions of $(\partial_t+P)(\partial_t^ku)=\partial_t^kf$\ on $I\times U$, knowing that $\partial_t^kf\in L^\infty_{\scaleto{\mbox{loc}}{5pt}}(I\times U)$\ for $1\leq k\leq n$, then proves the claim that the time derivatives of $u$\ are locally in $L^\infty(I\times U)$. Here the fact that the time derivatives of $u$\ are local weak solutions is guaranteed by Theorem 4.1 in \cite{L2}. Alternatively, one can show directly that the approximate sequence $\{\overline{\psi}\widetilde{u}_\tau\}_\tau$\ is Cauchy in $W^{n,\infty}(I\rightarrow L^\infty(X))$. We choose the first approach to simplify the proof, only show that $u$\ itself is locally bounded.

We now define precisely the functions in (\ref{approxseqdef}). For any $J\times V\Subset I\times U$, pick two nice product cut-off functions $\overline{\eta}(t,x)=\eta(x)l(t)$\ and $\overline{\psi}(t,x)=\psi(x)w(t)$\ satisfying that
\begin{itemize}
    \item $\overline{\eta}\equiv 1$\ on $J_{\overline{\eta}}\times V_{\overline{\eta}}$, $\mbox{supp}\{\overline{\eta}\}\subset I_{\overline{\eta}}\times U_{\overline{\eta}}$;
    \item $\overline{\psi}\equiv 1$\ on $J_{\overline{\psi}}\times V_{\overline{\psi}}$, $\mbox{supp}\{\overline{\psi}\}\subset I_{\overline{\psi}}\times U_{\overline{\psi}}$;
    \item the open sets satisfy
\begin{eqnarray}
J\times V\Subset J_{\overline{\psi}}\times V_{\overline{\psi}}\Subset I_{\overline{\psi}}\times U_{\overline{\psi}}\Subset J_{\overline{\eta}}\times V_{\overline{\eta}}\Subset I_{\overline{\eta}}\times U_{\overline{\eta}}\Subset I\times U.
\end{eqnarray}
\end{itemize}
The function $\rho_\tau$\ in (\ref{approxseqdef}) is defined as $\rho_\tau(t):=\frac{1}{\tau}\rho(\frac{t}{\tau})$, for any $\tau>0$, where $\rho:\mathbb{R}\rightarrow \mathbb{R}_{\geq 0}$\ is some function in $C_c^\infty(1,2)$\ that satisfies $\int_\mathbb{R}\rho = 1$. In (\ref{approxseqdef}), for each $s\in I$, $\widetilde{u}_\tau(s,x)$\ is only nonzero when $\tau>0$\ is small enough.

To prove that $\{\overline{\psi}\widetilde{u}_\tau\}_{\tau>0}$\ is Cauchy in $L^\infty(I\times X)$, which then proves Theorem \ref{mainthm0} as in the discussion above, it suffices to prove the following statements
\begin{itemize}
    \item[(1)] $\overline{\psi}\widetilde{u}_\tau\in L^\infty(I\times X)$\ for any $\tau>0$;
    \item[(2)] $\sup_{0<\tau<1}\left|\left|\partial_\tau(\overline{\psi}\widetilde{u}_\tau)\right|\right|_{L^\infty(I\times X)}<\infty$.
\end{itemize}
The first statement holds since for any fixed $\tau>0$, by the ultracontractivity property (\ref{globalultra}) of the semigroup,
\begin{eqnarray}
\label{prepest}
\lefteqn{\left|\left|\overline{\psi}\widetilde{u}_\tau\right|\right|_{L^\infty(I\times X)}=\esssup_{s\in I}\left|\left|\overline{\psi}(s,\cdot)\int_I\rho_\tau(s-t)H_{s-t}(\overline{\eta}^tu^t)\,dt\right|\right|_{L^\infty(X)}}\notag\\
&\leq& ||\overline{\psi}||_{L^\infty(I\times X)}\cdot \esssup_{s\in I}\int_I\rho_\tau(s-t)\left|\left|H_{s-t}(\overline{\eta}^tu^t)\right|\right|_{L^\infty(X)}\,dt\notag\\
&\leq& ||\overline{\psi}||_{L^\infty(I\times X)}\cdot \esssup_{s\in I}\int_I\rho_\tau(s-t)e^{M(s-t)}\left|\left|\overline{\eta}^tu^t\right|\right|_{L^2(X)}\,dt\notag\\
&\leq& ||\overline{\psi}||_{L^\infty(I\times X)}||\rho_\tau||_{L^\infty(\mathbb{R})}e^{M(\tau)}|I|^{1/2}\left|\left|\overline{\eta}u\right|\right|_{L^2(I\times X)}<\infty.
\end{eqnarray}
To prove the second statement, we start with computing $\partial_\tau\widetilde{u}_\tau$. Note that $\partial_\tau\rho_\tau(t)=-\partial_t\overline{\rho}_\tau(t)$, where $\overline{\rho}_\tau(t):=\frac{t}{\tau^2}\rho(\frac{t}{\tau})$. Thus
\begin{eqnarray*}
	\partial_\tau \widetilde{u}_\tau(s,x)
	%=\int_I\partial_\tau\rho_\tau(s-t)H_{s-t}(\overline{\eta}^tu^t)(x)\,dt}\\
	=\int_I\partial_t\overline{\rho}_\tau(s-t)\cdot H_{s-t}(\overline{\eta}^tu^t)(x)\,dt.
\end{eqnarray*}
On the other hand,
\begin{eqnarray*}
\lefteqn{\left|\left|\partial_\tau(\overline{\psi}\widetilde{u}_\tau)\right|\right|_{L^\infty(I\times X)}=\esssup_{s\in I}\left|\left|\partial_\tau(\overline{\psi}\widetilde{u}_\tau)\right|\right|_{L^\infty(X)}}\\
&=&\esssup_{s\in I}\sup_{\left|\left|\varphi\right|\right|_{L^1(X)}\leq 1} \int_X\overline{\psi}(s,x)\partial_\tau\widetilde{u}_\tau(s,x)\cdot \varphi(x)\,dm.
\end{eqnarray*}
After plugging in the expression for $\partial_\tau\widetilde{u}_\tau$, by the Fubini Theorem and the self-adjointness of the semigroup, we have
\begin{eqnarray}
\lefteqn{\left|\left|\partial_\tau(\overline{\psi}\widetilde{u}_\tau)\right|\right|_{L^\infty(I\times X)}}\notag\\
&=&\esssup_{s\in I}\sup_{\left|\left|\varphi\right|\right|_{L^1(X)}\leq 1} \int_I\int_X\overline{\eta}^tu^t\cdot \partial_t\overline{\rho}_\tau(s-t)w(s)H_{s-t}(\psi\varphi)\,dmdt\notag\\
&=&\esssup_{s\in I}\sup_{\left|\left|\varphi\right|\right|_{L^1(X)}\leq 1}\left\{\int_I\int_X\overline{\eta}^tu^t\cdot \partial_t[\overline{\rho}_\tau(s-t)w(s)H_{s-t}](\psi\varphi)\,dmdt\right.\notag\\
&&-\left.\int_I\int_X\overline{\eta}^tu^t\cdot \overline{\rho}_\tau(s-t)w(s)\partial_tH_{s-t}(\psi\varphi)\,dmdt\right\}.\label{split}
\end{eqnarray}
Let $v_\tau(s,t,x)$\ denote
\begin{eqnarray*}
v_\tau(s,t,x):=\overline{\rho}_\tau(s-t)w(s)\partial_tH_{s-t}(\psi\varphi)(x).
\end{eqnarray*}
Since $u$\ is a local weak solution of $(\partial_t+P)u=f$\ on $I\times U$, comparing (\ref{split}) and (\ref{localweaksol}) with $\varphi=\overline{\eta}v_\tau$\ leads to that
\begin{eqnarray*}
\left|\left|\partial_\tau(\overline{\psi}\widetilde{u}_\tau)\right|\right|_{L^\infty(I\times X)}=\esssup_{s\in I}\sup_{\left|\left|\varphi\right|\right|_{L^1(X)}}\left\{A_\tau(s,\tau)+B_\tau(s,\tau)+C_\tau(s,\tau)\right\},
\end{eqnarray*}
where
\begin{eqnarray*}
	A_\tau(s,\varphi)
	=-\int_I\int_Xu(t,x)\cdot \partial_t\left[\overline{\eta}(t,x)\right]\cdot v_\tau(s,t,x)\,dm(x)dt,
\end{eqnarray*}
\begin{eqnarray*}
	B_\tau(s,\varphi)
	=-\int_I\int_Xd\Gamma(\overline{\eta}^tu^t,\, v_\tau^{s,t})\,dt+\int_I\int_Xd\Gamma(u^t,\,\overline{\eta}^tv_\tau^{s,t})\,dt,
\end{eqnarray*}
\begin{eqnarray*}
	C_\tau(s,\varphi)= -\int_I\int_X f(t,x)\cdot \overline{\eta}(t,x)v_\tau(s,t,x)\,dm(x)dt.
\end{eqnarray*}
The killing measure parts in $B_\tau$\ are cancelled. Note that by (\ref{globalultra}) (global ultracontractivity of $H_t$), for any fixed $\tau>0$, $s\in I$, it is straightforward to verify that $v_\tau^s\in C^\infty(I\rightarrow \mathcal{F})$. By Lemma 3.5 in \cite{L2}, $\overline{\eta}v_\tau^s\in \mathcal{F}_c(I\times U)$.

We now estimate $A_\tau,B_\tau,C_\tau$. For
	\begin{eqnarray*}
		A_\tau(s,\varphi)
		=-\int_I\int_Xu(t,x)\cdot \partial_t[\overline{\eta}(t,x)]\cdot w(s)\overline{\rho}_\tau(s-t)H_{s-t}(\psi\varphi)(x)\,dm(x)dt,
	\end{eqnarray*}
	observe that when $\tau<c_0:=\min\left\{\frac{d(J_{\overline{\eta}}^c,\,I_{\overline{\psi}})}{2},\,\frac{1}{2}\right\}$, the product $\overline{\rho}_\tau(s-t)\cdot \partial_tl(t)\cdot w(s)\equiv 0$\ (recall that $\overline{\eta}(t,x)=l(t)\eta(x)$). Hence $A_\tau$\ is zero for $\tau$\ small. When $\tau\geq c_0$, note that
	\begin{eqnarray*}
	\lefteqn{\sup_{s,t\in I}\left\{\bar{\rho}_\tau(s-t)\left|\left|H_{s-t}(\psi\varphi)\right|\right|_{L^2(X)}\right\}}\\
	&\leq& \sup_{s,t\in I}\left\{\frac{s-t}{\tau^2}\rho\left(\frac{s-t}{\tau}\right)\left|\left|H_{s-t}\right|\right|_{L^1(X)\rightarrow L^2(X)}\left|\left|\psi\varphi\right|\right|_{L^1(X)}\right\}\\
	&\leq& \frac{2}{\tau}\left|\left|\rho\right|\right|_{L^\infty}e^{M(\tau)}\left|\left|\psi\varphi\right|\right|_{L^1(X)}\leq 2\left|\left|\psi\right|\right|_{L^\infty}\left|\left|\rho\right|\right|_{L^\infty}\frac{e^{M(c_0)}}{c_0}\left|\left|\varphi\right|\right|_{L^1(X)},
	\end{eqnarray*}
	hence
	\begin{eqnarray*}
		\lefteqn{|A_\tau(s,\varphi)|}\\
		&\leq& \left|\left|w\right|\right|_{L^\infty}\left|\left|\overline{\eta}\right|\right|_{C^1(I\rightarrow L^\infty(X))}\int_{I_{\overline{\eta}}}\ \left|\left|u^t\right|\right|_{L^2\left(U_{\overline{\eta}}\right)}\cdot\bar{\rho}_\tau(s-t)\left|\left|H_{s-t}(\psi\varphi)\right|\right|_{L^2(X)}\,dt\\
		&\leq& C(\overline{\eta},\overline{\psi},\rho)|I|^{1/2}\frac{e^{M(c_0)}}{c_0}\left|\left|\varphi\right|\right|_{L^1\left(X\right)}\left|\left|\overline{\Psi}u\right|\right|_{{L^2(I\times X)}}.
	\end{eqnarray*} 
	Here $\overline{\Psi}$\ is any nice product cut-off function supported in $I\times U$\ and equals $1$\ on the support of all other cut-off functions introduced in this proof. As $c_0=\min\left\{\frac{d(J_{\overline{\eta}}^c,\,I_{\overline{\psi}})}{2},\,\frac{1}{2}\right\}$\ depends only on $\overline{\eta},\overline{\psi}$, we conclude that for some constant $C_A(\overline{\eta},\overline{\psi},\rho)>0$,
	\begin{eqnarray*}
		\sup_{0<\tau<1}\esssup_{s\in I}\sup_{\left|\left|\varphi\right|\right|_{L^1(X)}\leq 1}|A_\tau(s, \varphi)|\leq C_A(\overline{\eta},\overline{\psi},\rho)\left|\left|\overline{\Psi}u\right|\right|_{L^2(I\times X)}.
	\end{eqnarray*}
	Next we estimate
	\begin{eqnarray*}
	B_\tau(s,\varphi)
	=-\int_I\int_Xd\Gamma(\overline{\eta}^tu^t,\, v_\tau^{s,t})\,dt+\int_I\int_Xd\Gamma(u^t,\,\overline{\eta}^tv_\tau^{s,t})\,dt.
\end{eqnarray*}
	By the strong locality of the energy measure $d\Gamma$, for any $t\in I$, on $V_{\overline{\eta}}$\ where $\eta\equiv 1$,
	\begin{eqnarray*}
	1_{V_{\overline{\eta}}}\,d\Gamma(\overline{\eta}^t u^t,\, v_\tau^{s,t})=1_{V_{\overline{\eta}}}\,d\Gamma(u^t,\,\overline{\eta}^tv_\tau^{s,t}).
	\end{eqnarray*}
	Therefore,
	\begin{eqnarray*}
		B_\tau(s,\varphi)
		=-\int_I\int_Xd\Gamma(\overline{\eta}^tu^t,\,\Phi v_\tau^{s,t})\,dt+\int_I\int_Xd\Gamma(u^t,\,\Phi\overline{\eta}^tv_\tau^{s,t})\,dt,
	\end{eqnarray*}
	where $\Phi$\ is a nice cut-off function which intuitively is ``hollow-shaped'', being $0$\ inside $V_{\overline{\eta}}$\ and being $1$\ over an open set that covers where $\eta$\ is nonzero and not $1$. More precisely, $\Phi\equiv 1$\ on $V_{\Phi}$\ and $\mbox{supp}\{\Phi\}\subset U_{\Phi}$, for sets $V_\Phi:=V''-\overline{U'}$, $U_\Phi:=U''-\overline{V'}$, where $V', U', V'', U''$\ satisfy
	\begin{eqnarray*}
	V_{\overline{\psi}}\Subset U_{\overline{\psi}}\Subset V'\Subset U'\Subset V_{\overline{\eta}}\Subset U_{\overline{\eta}}\Subset V''\Subset U''\Subset U.
	\end{eqnarray*}
	In particular, the supports of $\psi$\ and $\Phi$\ are disjoint.
	
	Now we use the Cauchy-Schwartz inequality for $d\Gamma$\ to get
	\begin{eqnarray*}
		\lefteqn{|B_\tau(s,\varphi)|
		\leq \left(\left|\left|\overline{\eta}u\right|\right|_{L^2(I\rightarrow \mathcal{F})}+ \left|\left|\overline{\Psi}u\right|\right|_{L^2(I\rightarrow \mathcal{F})}\right)\times}\\
	    &&\times \left[\left(\int_I\mathcal{E}(\Phi v_\tau^{s,t},\,\Phi v_\tau^{s,t})\,dt\right)^{1/2}+\left(\int_I\mathcal{E}(\Phi \overline{\eta}^tv_\tau^{s,t},\,\Phi \overline{\eta}^tv_\tau^{s,t})\,dt\right)^{1/2}\right].
	\end{eqnarray*}
	The estimates for the two terms in the square bracket are almost identical, we only give estimates for the second term as an example. Using the gradient inequality (\ref{gradineq}), and in the last line using the $L^\infty$\ off-diagonal upper bound (\ref{globalGaussian}), we have
	\begin{eqnarray*}
		\lefteqn{\int_I\mathcal{E}(\Phi \overline{\eta}^tv_\tau^{s,t},\,\Phi \overline{\eta}^tv_\tau^{s,t})\,dt}\\
		&\leq&  2\int_I\left|\int_X(\Phi \overline{\eta}^t)^2v_\tau^{s,t}\cdot Pv_\tau^{s,t}\,dm\right|dt+2C_2\int_I\int_X1_\Phi v_\tau^{s,t}\cdot v_\tau^{s,k}\,dmdt\\
		&=& 2\int_I\left|\int_X(\Phi \overline{\eta}^t)^2v_\tau^{s,t}\cdot w(s)\overline{\rho}_\tau(s-t)PH_{s-t}(\psi\varphi)\,dm\right|dt\\
		&&+2C_2\int_I\int_X1_\Phi v_\tau^{s,t}\cdot w(s)\overline{\rho}_\tau(s-t)H_{s-t}(\psi\varphi)\,dmdt\\
		&\leq& 2\left|\left|w\right|\right|_{L^\infty}\left|\left|\rho\right|\right|_{L^\infty}\sup_{\tau<r<2\tau}\left\{\frac{1}{\tau}e^{-G(1,r)}\right\}\int_I \left|\left|(\Phi\overline{\eta}^t)^2v_\tau^{s,t}\right|\right|_{L^1(X)}\left|\left|\psi\varphi\right|\right|_{L^1\left(X\right)}\,dt\\
		&&+2C_2\left|\left|w\right|\right|_{L^\infty}\left|\left|\rho\right|\right|_{L^\infty}\sup_{\tau<r<2\tau}\left\{\frac{1}{\tau}e^{-G(0,r)}\right\}\int_I \left|\left|1_\Phi v_\tau^{s,t}\right|\right|_{L^1(X)}\left|\left|\psi\varphi\right|\right|_{L^1\left(X\right)}\,dt.
	\end{eqnarray*}
	The term $\left|\left|(\Phi\overline{\eta}^t)^2v_\tau^{s,t}\right|\right|_{L^1(X)}$\ in the integral is bounded by
	\begin{eqnarray*}
		\lefteqn{\left|\left|(\Phi\overline{\eta}^t)^2v_\tau^{s,t}\right|\right|_{L^1(X)}}\\
		&\leq& \left|\left|w\right|\right|_{L^\infty}\left|\left|\rho\right|\right|_{L^\infty}\sup_{\tau<r<2\tau}\left\{\frac{1}{\tau}e^{-G(0,r)}\right\}\cdot \left|\left|\psi\varphi\right|\right|_{L^1(X)}\left|\left|(\Phi\overline{\eta}^t)^2\right|\right|_{L^1(X)};
	\end{eqnarray*}
	the term $\left|\left|1_\Phi v_\tau^{s,t}\right|\right|_{L^1(X)}$\ is similarly bounded.
	Hence
	\begin{eqnarray*}
		\lefteqn{\int_I\mathcal{E}(\Phi \overline{\eta}^tv_\tau^{s,t},\,\Phi \overline{\eta}^tv_\tau^{s,t})\,dt}\\
		&\leq& C(\overline{\psi},\overline{\eta},\rho,\Phi)\sup_{\tau<r<2\tau}\left\{\frac{1}{\tau}\left(e^{-G(0,r)}+e^{-G(1,r)}\right)\right\}^2\left|\left|\varphi\right|\right|_{L^1\left(X\right)}^2,
	\end{eqnarray*}
	where
	\begin{eqnarray*}
	C(\overline{\psi},\overline{\eta},\rho,\Phi):=2(1+C_2)||\overline{\psi}||^2_{L^\infty(I\times X)}||\rho||^2_{L^\infty}\left(1+\left|\left|(\Phi\overline{\eta})^2\right|\right|_{L^1(I\times X)}\right),
	\end{eqnarray*}
	and $C_2$\ is associated with $\Phi\overline{\eta}$.
%	We remark that the term $\left|\left|\varphi\right|\right|_{L^1(X)}^2$\ in the upper bound makes it not applicable to take the test functions in space and time (i.e. $\varphi\in L^1(I\times X)$) as in the proof in \cite{L2} for the $L^2$\ time regularity of local weak solutions. So here we consider $\left|\left|\partial_\tau(\overline{\psi}\widetilde{u}_\tau)\right|\right|_{L^\infty(X)}$\ first, and then consider the essential supremum in $s\in I$\ of this quantity.\\[0.1in]
	By (\ref{mcontrolpoly}),
	\begin{eqnarray*}
	\sup_{0<\tau<1}\sup_{\tau<r<2\tau}\left\{\frac{1}{\tau}\left(e^{-G(0,r)}+e^{-G(1,r)}\right)\right\}^2<+\infty,
	\end{eqnarray*}
	hence for some constant $C_B(\overline{\psi},\overline{\eta},\rho,\Phi)>0$,
	\begin{eqnarray*}
		\lefteqn{\sup_{0<\tau<1}\esssup_{s\in I}\sup_{\left|\left|\varphi\right|\right|_{L^1(X)}\leq 1}|B_\tau(s,\varphi)|}\\
		&\leq& C_B(\overline{\psi},\overline{\eta},\rho,\Phi)\cdot\left(\left|\left|\overline{\eta}u\right|\right|_{L^2(I\rightarrow \mathcal{F})}+ \left|\left|\overline{\Psi}u\right|\right|_{L^2(I\rightarrow \mathcal{F})}\right).
	\end{eqnarray*}

	Finally, by facts that $||H_{s-t}||_{L^1(X)\rightarrow L^1(X)}\leq 1$\ and $1\leq \int_I\overline{\rho}_\tau(s-t)\,dt\leq 2$, the term $|C_\tau(s,\varphi)|$\ satisfies
	\begin{eqnarray*}
		\lefteqn{|C_\tau(s, \varphi)|}\\
		&=& \left|\int_I\int_X f(t,x)\cdot \overline{\eta}(t,x)w(s)\overline{\rho}_\tau(s-t)H_{s-t}(\psi\varphi)(x)\,dm(x)dt\right|\\
		%&\leq& ||w||_{L^\infty(I)}||\overline{\eta}f||_{L^\infty(I\times X)}||H_{s-t}||_{L^1(X)\rightarrow L^1(X)}||\psi\varphi||_{L^1(X)}\int_I\overline{\rho}_\tau(s-t)\,dt\\
		&\leq& 2||\overline{\psi}||_{L^\infty(I\times X)}||\overline{\eta}f||_{L^\infty(I\times X)}||\varphi||_{L^1(X)}.
	\end{eqnarray*}
Hence for $C_C(\overline{\psi}):=2||\overline{\psi}||_{L^\infty(I\times X)}$,
	\begin{eqnarray*}
		\sup_{0<\tau<1}\esssup_{s\in I}\sup_{\left|\left|\varphi\right|\right|_{L^1(X)}\leq 1}|C_\tau(s, \varphi)|\leq C_C(\overline{\psi})\left|\left|\overline{\eta}f\right|\right|_{L^\infty(I\times X)}.
	\end{eqnarray*}
	Putting the estimates for $A_\tau$, $B_\tau$, and $C_\tau$\ together, we get
\begin{eqnarray}
\label{derivativeest}
		\lefteqn{\sup_{0<\tau<1}\left|\left|\partial_\tau(\overline{\psi}\widetilde{u}_\tau)\right|\right|_{L^\infty(I\times X)}}\notag\\
		&\leq& \left(C_A(\overline{\eta},\overline{\psi},\rho)+C_B(\overline{\psi},\overline{\eta},\rho,\Phi)+C_C(\overline{\psi})\right)\notag\\
		&&\cdot \left(\left|\left|\overline{\eta}u\right|\right|_{L^2(I\rightarrow \mathcal{F})}+ \left|\left|\overline{\Psi}u\right|\right|_{L^2(I\rightarrow \mathcal{F})}+\left|\left|\overline{\eta}f\right|\right|_{L^\infty(I\times X)}\right)<\infty.
\end{eqnarray}
Hence $\{\overline{\psi}\widetilde{u}_\tau\}_{\tau>0}$\ is Cauchy in $L^\infty(I\times X)$. By the discussion at the beginning of the proof, it follows from (\ref{prepest})(\ref{derivativeest}) that
\begin{eqnarray}
\label{keyest}
\lefteqn{\left|\left|\overline{\psi}u\right|\right|_{L^\infty(I\times X)}}\notag\\
&\hspace{-.3in}\leq& \hspace{-.15in}C(\overline{\eta},\overline{\psi},\Phi,\rho) \left(\left|\left|\overline{\eta}u\right|\right|_{L^2(I\rightarrow \mathcal{F})}+ \left|\left|\overline{\Psi}u\right|\right|_{L^2(I\rightarrow\mathcal{F})}+\left|\left|\overline{\eta}f\right|\right|_{L^\infty(I\times X)}\right)
\end{eqnarray}
for some $C(\overline{\eta},\overline{\psi},\Phi,\rho)>0$. Hence $u\in L^\infty(J\times V)$. Since $J\times V\Subset I\times U$\ is arbitrary, $u\in L^\infty_{\scaleto{\mbox{loc}}{5pt}}(I\times U)$. By applying this result to $\partial_t^ku$\ as local weak solutions of $(\partial_t+P)\partial_t^ku=\partial_t^kf$, $1\leq k\leq n$, we conclude that $u$\ is locally in $W^{n,\infty}(I\rightarrow L^\infty(U))$.
\end{proof}
Note that in Theorem \ref{mainthm0}, by the global ultracontractivity condition, the semigroup $H_t$\ admits a density function, i.e., there is a measurable function $h(t,x,y)$\ on $(0,T)\times X\times X$\ such that for any $g\in L^2(X)$,
\begin{eqnarray*}
H_tg(x)=\int_Xh(t,x,y)g(y)\,dm(y)
\end{eqnarray*}
for any $t\in (0,T)$\ and a.e. $x$. Cf. \cite{DunfordPettis}. 
%Moreover, for any $t\in (0,T)$, $h(t,x,y)$\ is $L^2$\ in either spatial variable, and $x\mapsto \int_X h(t,x,y)^2\,dm(y)\in L^\infty(X)$. 
The global ultracontractivity condition in Theorem \ref{mainthm0} is equivalent to the condition
\begin{eqnarray*}
\esssup_{x,y\in X}h(t,x,y)\leq e^{M_1(t)},
\end{eqnarray*}
where $M_1(t)\simeq M(t)$.

In Theorem \ref{mainthm0}, if moreover the density $h(t,x,y)$\ is continuous on some $(0,c)\times V\times V$\ where $(0,c)\subset (0,T)$\ and $V\subset U$, then local weak solutions are continuous on $I\times V$\ as well. This is the following corollary.
\begin{corollary}
	\label{continuitythm}	
	Assume the hypotheses in Theorem \ref{mainthm0}. Suppose the density $h(t,x,y)$\ of the semigroup $(H_t)_{t>0}$\ is continuous on $(0,c)\times V\times V$\ for some $(0,c)\subset (0,T)$\ and $V\Subset U$. Then any local weak solution $u$\ of the heat equation $(\partial_t+P)u=f$\ on $I\times U$\ is continuous on $I\times V$.
\end{corollary}
\begin{remark}
More precisely, $u$\ being continuous on $I\times V$\ refers to that $u$\ has a continuous version on $I\times V$. In the following we adopt this convention.
\end{remark}
\begin{proof}
For any $J\times V'\Subset I\times V$, as in the proof of Theorem \ref{mainthm0}, we introduce the same nice product cut-off functions $\overline{\eta}$, $\overline{\psi}$\ accordingly. In this setting,
\begin{eqnarray*}
J\times V'\Subset J_{\overline{\psi}}\times V_{\overline{\psi}}\Subset I_{\overline{\psi}}\times U_{\overline{\psi}}\Subset J_{\overline{\eta}}\times V_{\overline{\eta}}\Subset I_{\overline{\eta}}\times U_{\overline{\eta}}\Subset I\times V.
\end{eqnarray*}
By Theorem \ref{mainthm0}, the approximate sequence $\{\overline{\psi}\widetilde{u}_\tau\}_{\tau>0}$\ converges to $\overline{\psi}u$\ in $L^\infty(I\times X)$, hence it suffices to show that each function $\overline{\psi}\widetilde{u}_\tau$\ is continuous on $J\times V'$. Take the version of $\widetilde{u}_\tau$\ given pointwisely by
\begin{eqnarray*}
\widetilde{u}_\tau(s,x)=\int_I\rho_\tau(s-t)\int_Xh(s-t,x,y)(\overline{\eta}^tu^t)(y)\,dm(y)dt.
\end{eqnarray*}
For any fixed $\tau\in (0,\frac{c}{2})$, for any two pairs of $(s,x),(s',x')\in J\times V'$, let $J_{s,\tau}:=(s-2\tau,s-\tau)$, $J_{s',\tau}:=(s'-2\tau,s'-\tau)$, then
	\begin{eqnarray}
		\lefteqn{|\widetilde{u}_\tau(s,x)-\widetilde{u}_\tau(s',x')|}\notag\\
		&\leq& \left|\int_I1_{J_{s,\tau}}(t)(\rho_\tau(s-t)-\rho_\tau(s'-t))H_{s-t}(\overline{\eta}^tu^t)(x)\,dt\right|+\left|\int_I\rho_\tau(s'-t)\times\right.\notag\\
		&&\left.\times\left(1_{J_{s,\tau}}(t)H_{s-t}(\overline{\eta}^tu^t)(x)-1_{J_{s',\tau}(t)}H_{s'-t}(\overline{\eta}^tu^t)(x')\right)dt\right|.\label{continuityest}
	\end{eqnarray}
	Since
	\begin{eqnarray*}
	|H_{s-t}(\overline{\eta}^tu^t)(x)|=|\int_Xh(s-t,x,y)\overline{\eta}(t,y)u(t,y)\,dm(y)dt|\leq \left|\left|\overline{\eta}u\right|\right|_{L^\infty(I\times X)},
	\end{eqnarray*}
	the first term in (\ref{continuityest}) is bounded by
	\begin{eqnarray*}
		\left|\left|\overline{\eta}u\right|\right|_{L^\infty(I\times X)}|I| \cdot \sup_{t\in I}|\rho_\tau(s-t)-\rho_\tau(s'-t)|, 
	\end{eqnarray*}
	which tends to $0$\ as $|s-s'|$\ tends to $0$\ as $\rho_\tau$\ is uniformly continuous and $\left|\left|\overline{\eta}u\right|\right|_{L^\infty(I\times X)}<\infty$\ by Theorem \ref{mainthm0}.
	
	Split the second term in (\ref{continuityest}) into integrals on $J_{s,\tau}\cap J_{s',\tau}$, $J_{s,\tau}\setminus J_{s',\tau}$, and $J_{s',\tau}\setminus J_{s,\tau}$. The first part satisfies
	\begin{eqnarray*}
		\lefteqn{\left|\int_{J_{s,\tau}\cap J_{s',\tau}}\rho_\tau(s'-t)\left(1_{J_{s,\tau}}(t)H_{s-t}(\overline{\eta}^tu^t)(x)-1_{J_{s',\tau}}(t)H_{s'-t}(\overline{\eta}^tu^t)(x')\right)dt\right|}\\
		%&=& \left|\int_I\rho_\tau(s'-t)\int_X \left(1_{J_{\tau,s}}(t)h(s-t,x,y)-1_{J_{s',\tau}}(t)h(s'-t,x',y)\right)\overline{\eta}(t,y)u(t,y)\,dm(y)dt\right|\\
		&\leq&\left|\left|\overline{\eta}u\right|\right|_{L^\infty(I\times X)}m(U_{\overline{\eta}}) \sup_{\substack{\tau<s-t,\,s'-t<2\tau,\\y\in U_{\overline{\eta}}}}\left|h(s-t,x,y)-h(s'-t,x',y)\right|.
	\end{eqnarray*} 
	By the uniform continuity of the heat kernel $h(r,x,y)$\ on $(\tau,2\tau)\times V'\times U_{\overline{\eta}}\Subset (0,c)\times V\times V$, the upper bound tends to $0$\ as $|s-s'|$\ tends to $0$\ and the distance between $x$\ and $x'$\ tends to $0$.
	
	The other two parts are both bounded by $\left|\left|\overline{\eta}u\right|\right|_{L^\infty(I\times X)}|s-s'|$, which tends to $0$\ as $|s-s'|$\ tends to $0$.

	Hence the functions in the approximate sequence, $\overline{\psi}\widetilde{u}_\tau$\ for any $0<\tau<1$, are continuous on $J\times V'$. It follows that the limit $\overline{\psi}u$\ is continuous on $J\times V'$. By varying $J,V'$, we conclude that $u$\ is continuous on $I\times V$.
\end{proof}
\begin{remark}
When $X$\ is a compact group and the Dirichlet form corresponds to a convolution semigroup, the boundedness of the kernel implies the continuity of it. In such cases, the condition in Corollary \ref{continuitythm} is automatically satisfied once conditions in Theorem \ref{mainthm0} are met.
\end{remark}
\section{Generalized framework for the main theorem}
\label{frameworksection}
\setcounter{equation}{0}
There are many settings beyond Dirichlet form ones where a class of functions are local weak solutions of some heat equation when restricted to some local open set. One natural setting, a direct generalization of the Dirichlet form setting, is when globally there is a bilinear form that locally agrees with a Dirichlet form. Our goal in this section is to describe a general framework to draw Theorem \ref{mainthm0} type conclusions for heat equation solutions in broader settings. After revisiting the concept of local ultracontractivity, we first state the general framework (Theorem \ref{generalthm}). Then we make precise the special setting of bilinear forms locally represented by Dirichlet forms and state Theorem \ref{mainthm} (the version of Theorem \ref{generalthm} in that setting). We delay the proofs until the end of this section. The proofs for Theorems \ref{generalthm} and \ref{mainthm} are identical in nature, we only present the proof for Theorem \ref{mainthm}. Before that, we give two other examples to which the general framework (Theorem \ref{generalthm}) applies.
\subsection{Statement of the generalized framework}
In Section \ref{globalsection} we defined two properties that a semigroup may or may not satisfy - the global ultracontractivity property (\ref{globalultra}), and the $L^\infty$\ off-diagonal upper bound (\ref{globalGaussian})(\ref{mcontrolpoly}). As noted before, the global ultracontractivity property is equivalent to the existence of a globally (essentially) bounded density kernel. From this perspective it is evident that a more natural assumption is the existence of a locally bounded kernel, or the local ultracontractivity of the semigroup. In the next section we discuss the implication from one condition to the other. Here we first specify the property of local ultracontractivity. Such properties can be defined for general operators, here we restrict ourselves to semigroups corresponding to Dirichlet forms. %By a heat semigroup below we mean a strongly continuous Markov semigroup.
\begin{definition}[local ultracontractivity]
\label{localultra}
Let $(X,m,\mathcal{E},\mathcal{F})$\ be a Dirichlet space with corresponding heat semigroup $(H_t)_{t>0}$\ on $L^2(X,m)$. For any precompact open subset $\Omega\subset X$, the semigroup is said to be locally ultracontractive on $\Omega$, if there exists some $T_\Omega>0$\ and some continuous nonincreasing function $M_\Omega:(0,T_\Omega)\rightarrow \mathbb{R}_+$, such that
	\begin{eqnarray}
	\label{localultraassumption}
	\left|\left|H_t\right|\right|_{L^2(\Omega)\rightarrow L^\infty(\Omega)}\leq e^{M_\Omega(t)}.
	\end{eqnarray}
$(H_t)_{t>0}$\ is said to satisfy the local ultracontractivity property, if it is locally untracontractive on any $\Omega\Subset X$.
\end{definition}
%In other words, local ultracontractivity means the upper bound for the operator norm depends on 
\begin{remark}
Unlike global ultracontractivity where the ``$2\rightarrow \infty$'' and ``$1\rightarrow \infty$'' conditions are equivalent (i.e.,
$||H_t||_{L^2(X)\rightarrow L^\infty(X)}$\ and $||H_t||_{L^1(X)\rightarrow L^\infty(X)}$\ being finite imply each other with explicit relations between the upper bounds), when the whole space $X$\ is replaced by some subset $\Omega\Subset X$\ in the above two operator norms, it is no longer clear if $||H_t||_{L^2(\Omega)\rightarrow L^\infty(\Omega)}<\infty$\ implies $||H_t||_{L^1(\Omega)\rightarrow L^\infty(\Omega)}<\infty$\ in general. For example, the previous method no longer works because the semigroup is a global operator: while by self-adjointness we still have $||H_t||_{L^2(\Omega)\rightarrow L^\infty(\Omega)}=||H_t||_{L^1(\Omega)\rightarrow L^2(\Omega)}$, breaking $||H_t||_{L^1(\Omega)\rightarrow L^\infty(\Omega)}$\ as $||H_{t/2}H_{t/2}||_{L^1(\Omega)\rightarrow L^\infty(\Omega)}$\ only leads to the upper bound
\begin{eqnarray*}
||H_t||_{L^1(\Omega)\rightarrow L^\infty(\Omega)}\leq ||H_{t/2}||_{L^1(\Omega)\rightarrow L^2(X)}||H_{t/2}||_{L^2(X)\rightarrow L^\infty(\Omega)}.
\end{eqnarray*}
In the definition above we take the possibly weaker condition (\ref{localultraassumption}).
\end{remark}
%We remark that the $L^\infty$\ off-diagonal upper bound is by itself a local property (i.e. holds for any $V,W\Subset X$\ with $\overline{V}\cap\overline{W}=\emptyset$).
%\textcolor{blue}{\begin{assumption}[(local) $L^\infty$\ off-diagonal upper bound]
%\label{Gaussiandef}
%Let $(X,m,\mathcal{E},\mathcal{F})$\ be a Dirichlet space. Its corresponding heat semigroup $(H_t)_{t>0}$\ on $L^2(X,m)$\ is said to satisfy the $L^\infty$\ off-diagonal upper bound in some $\Omega\subset X$, if for any precompact open sets $V,W\Subset \Omega$\ with $\overline{V}\bigcap \overline{W}=\emptyset$, for any $n\in \mathbb{N}$, there exists some continuous, positive function $G(\Omega, V,W,n,t)=:G(n,t)$\ such that for any $v,w\in L^1(X)$\ with $\mbox{supp}\{v\}\subset V$, $\mbox{supp}\{w\}\subset W$, 
%		\begin{eqnarray}
%		\label{localGaussian}
%		\left|\int_X\partial_t^nH_tv\cdot w\,dm\right|\leq e^{-G(n,t)}\left|\left|v\right|\right|_{L^1}\left|\left|w\right|\right|_{L^1},
%		\end{eqnarray}
%		where $G(n,t)$\ satisfies that for any $a\geq 0$, any $n\in \mathbb{N}$,
%        \begin{eqnarray}
 %       \label{localmcontrolpoly}
%        \lim_{t\rightarrow 0}\frac{1}{t^a}e^{-G(n,t)}<\infty.
%       \end{eqnarray}
%\end{assumption}}

When $H_t$\ has a density kernel $h(t,x,y)$, the $L^\infty$\ off-diagonal upper bound (\ref{globalGaussian}) is equivalent to corresponding essential supremum bounds for $h(t,x,y)$
\begin{eqnarray*}
\esssup_{x\in V,\,y\in W}\left|\partial_t^nh(t,x,y)\right|\leq e^{-G(n,t)}.
\end{eqnarray*}

In practice, Theorem \ref{generalthm} below illustrates best how we draw conclusions of the local boundedness and continuity properties for heat equation solutions in more general settings.
\begin{theorem}[general framework]
\label{generalthm}
Let $X$\ be a metric measure space as before. Let $\widetilde{\mathcal{C}}$\ be a class of functions on $I\times X$, where $I\subset \mathbb{R}$. Let $(\mathcal{E}_\Omega,\mathcal{F}_\Omega)$\ be a symmetric local regular Dirichlet form on $L^2(\Omega,\mu_\Omega)$, where $\Omega\subset X$\ is an open subset and $\mu_\Omega$\ is some measure on $\Omega$. Assume that the Dirichlet form satisfies Assumption \ref{cutoff} (existence of nice cut-off functions); the associated semigroup $(H_t^\Omega)_{t>0}$\ is locally ultracontractive (\ref{localultraassumption}) and satisfies the $L^\infty$\ off-diagonal upper bound (\ref{globalGaussian})(\ref{mcontrolpoly}) on $\Omega$. Assume that all functions $u$\ in the class $\widetilde{\mathcal{C}}$\ have their restrictions on $I\times \Omega$\ being local weak solutions of the heat equation $(\partial_t+P_\Omega)u=f$, where $f$\ is locally in $W^{n,\infty}(I\rightarrow L^\infty(\Omega))$. Then all functions in $\widetilde{\mathcal{C}}$\ locally belong to $W^{n,\infty}(I\rightarrow L^\infty(\Omega))$.

Moreover, if $H_t^\Omega$\ admits a continuous density function on $(0,c)\times V\times V$\ for some interval $(0,c)\subset (0,T)$\ and open subset $V\subset\Omega$, then all functions in $\widetilde{\mathcal{C}}$\ are continuous on $I\times V$. Here $(0,T)$\ is the interval over which the ultracontractivity and off-diagonal bounds hold.
\end{theorem}
\subsection{Special case - bilinear form setting}
In this subsection we define what it means for a bilinear form to be locally represented by a Dirichlet form, and state the version of the main theorem in this setting. Theorems \ref{modelthm} and \ref{mainthm0} are included in this bilinear form setting.
\begin{definition}
Let $(X,m)$\ be a metric measure space. $(\mathcal{B},\widetilde{\mathcal{D}},\widetilde{\mathcal{D}}_0)$\ is called a bilinear triple, if
\begin{itemize}
    \item[(i)] $\widetilde{\mathcal{D}}_0\subset \widetilde{\mathcal{D}}\subset L^2_{\scaleto{\mbox{loc}}{5pt}}(X,m)$\ as vector spaces;
    \item[(ii)] $\mathcal{B}(f,g)$\ is defined for any $f\in \widetilde{\mathcal{D}}$, $g\in \widetilde{\mathcal{D}}_0$;
    \item[(iii)] $\mathcal{B}$\ is bilinear in the sense that for any fixed $f\in \widetilde{D}$, $\mathcal{B}(f,\cdot)$\ is linear on $\widetilde{\mathcal{D}}_0$; for any fixed $g\in \widetilde{\mathcal{D}}_0$, $\mathcal{B}(\cdot,g)$\ is linear on $\widetilde{\mathcal{D}}$.  
\end{itemize}
\end{definition}
\begin{definition}
\label{agreelocally}
Let $(\mathcal{B},\widetilde{\mathcal{D}},\widetilde{\mathcal{D}}_0)$\ be a bilinear triple. Let $\Omega$\ be an open subset of $X$. We say that $(\mathcal{B},\widetilde{\mathcal{D}},\widetilde{\mathcal{D}}_0)$\ is represented by a symmetric local regular Dirichlet form in $\Omega$, if there exist a measure $\mu_\Omega$\ on $\Omega$\ that is not necessarily the restriction of $m$\ on $\Omega$\ and a symmetric local regular Dirichlet form $(\mathcal{E}_\Omega,\mathcal{F}_\Omega)$\ on $L^2(\Omega,\mu_\Omega)$, such that 
\begin{itemize}
    \item[(i)] for any $f\in \widetilde{\mathcal{D}}$, the restriction $f\big|_\Omega\in \mathcal{F}_{\Omega,\,\scaleto{\mbox{loc}}{5pt}}$;
    \item[(ii)] $\mathcal{F}_{\Omega,\,c}\subset \widetilde{\mathcal{D}}_0$;
    \item[(iii)] for any $f\in \widetilde{\mathcal{D}}$, $g\in \mathcal{F}_{\Omega,\,c}$,
    \begin{eqnarray*}
    \mathcal{B}(f,g)=\mathcal{E}_\Omega\left(f\big|_\Omega,g\right).
    \end{eqnarray*}
\end{itemize}
\end{definition}
Here $\mathcal{F}_{\Omega,\,\scaleto{\mbox{loc}}{5pt}}=\mathcal{F}_{\Omega,\,\scaleto{\mbox{loc}}{5pt}}(\Omega)$, $\mathcal{F}_{\Omega,\,c}=\mathcal{F}_{\Omega,\,c}(\Omega)$\ denote the $\sim_{\scaleto{\mbox{loc}}{5pt}}$, $\sim_{c}$\ spaces associated with the Dirichlet form $(\mathcal{E}_\Omega,\mathcal{F}_\Omega)$.
\begin{remark}
It is clear that for any $U\Subset \Omega$, we also have $f\big|_U\in \mathcal{F}_{\Omega,\,\scaleto{\mbox{loc}}{5pt}}(U)$.
\end{remark}

In the setting of a global bilinear form on $X$\ locally represented by some symmetric local regular Dirichlet form, Theorem \ref{generalthm} can be restated as the theorem below.
\begin{theorem}
\label{mainthm}
Let $(X,m)$\ be a metric measure space. Let $\Omega\subset X$\ be an open subset and $I\subset \mathbb{R}$\ be an open interval. Let $(\mathcal{B},\widetilde{\mathcal{D}},\widetilde{\mathcal{D}}_0)$\ be a bilinear triple that is represented by a symmetric local regular Dirichlet form $(\mathcal{E}_\Omega,\mathcal{F}_\Omega)$\ in $\Omega$. The Dirichlet form $(\mathcal{E}_\Omega,\mathcal{F}_\Omega)$\ is on $L^2(\Omega,\mu_\Omega)$\ for some measure $\mu_\Omega$\ on $\Omega$. Let $(H_t^\Omega)_{t>0}$\ denote the semigroup associated with $(\mathcal{E}_\Omega,\mathcal{F}_\Omega)$. Suppose $(\mathcal{E}_\Omega,\mathcal{F}_\Omega)$\ satisfies Assumption \ref{cutoff} (existence of nice cut-off functions); the semigroup $(H_t^\Omega)_{t>0}$\ is locally ultracontractive (\ref{localultraassumption}), and satisfies the $L^\infty$\ off-diagonal upper bound (\ref{globalGaussian}) and (\ref{mcontrolpoly}) on $\Omega$. Fix any open subset $U\subset \Omega$. Let $f$\ be locally in  $W^{n,\infty}(I\rightarrow L^\infty(U,\mu_\Omega))$, where $n\in \mathbb{N}$. Here $\mu_\Omega$\ stands for the restriction of the measure $\mu_\Omega$\ on $U$. Then for any function $u$\ defined on $I\times X$, satisfying
\begin{itemize}
    \item for a.e. $t\in I$, $u(t,\cdot)\in \widetilde{\mathcal{D}}$;
    \item $u\big|_{I\times U}\in \mathcal{F}_{\Omega,\,\scaleto{\mbox{loc}}{5pt}}(I\times U)$;
    \item for any $\varphi\in C_c^\infty(I\rightarrow \mathcal{F}_{\Omega})\cap \mathcal{F}_{\Omega,\,c}(I\times U)$,
\begin{eqnarray*}
-\int_I\int_U u\cdot \partial_t\varphi\,d\mu_\Omega dt+\int_I\mathcal{B}(u,\varphi)\,dt=\int_I\int_U f\varphi\,d\mu_\Omega dt,
\end{eqnarray*}
\end{itemize}
$u$\ also locally belongs to $W^{n,\infty}(I\rightarrow L^\infty(U,\mu_\Omega))$.

Moreover, if $(H_t^\Omega)_{t>0}$\ admits a density kernel $h_\Omega(t,x,y)$\ that is continuous on $(0,c)\times V\times V$\ for some open subsets $(0,c)\subset (0,T)$\ and $V\subset U$, where $(0,T)$\ is the interval over which the ultracontractivity and off-diagonal bounds hold, then $u$\ is continuous on $I\times V$.
\end{theorem}
We first look at a familiar example viewed in this perspective.
\begin{example}[restriction form]
\label{restrictionformeg}
Let $(\mathcal{E},\mathcal{F})$\ be a symmetric local regular Dirichlet form on $X$. For any precompact open subset $\Omega\Subset X$, let $(\mathcal{E}^\Omega_0,\mathcal{F}^\Omega_0)$\ be the Dirichlet form obtained by considering the restriction of $\mathcal{E}$\ on $\mathcal{F}_c(\Omega)$\ first, then completing $\mathcal{F}_c(\Omega)$\ w.r.t. the $\mathcal{E}_1$\ norm to get $\mathcal{F}^\Omega_0$, the domain of $\mathcal{E}^\Omega_0$. It is known that the so-obtained form $(\mathcal{E}^\Omega_0,\mathcal{F}^\Omega_0)$\ is still symmetric, local, and regular. We call the Dirichlet form defined this way {\em the restriction form of $(\mathcal{E},\mathcal{F})$\ on $\Omega$}. In the terminology introduced in the definitions above, $(\mathcal{E},\mathcal{F}_{\scaleto{\mbox{loc}}{5pt}}(X),\mathcal{F}_c(X))$, $(\mathcal{E},\mathcal{F},\mathcal{F})$\ are both bilinear triples that are represented by the symmetric local regular Dirichlet form $(\mathcal{E}^\Omega_0,\mathcal{F}^\Omega_0)$\ in $\Omega$.

Moreover, if the Dirichlet form $(\mathcal{E},\mathcal{F})$\ satisfies Assumption \ref{cutoff} (existence of nice cut-off functions), then for any $V\subset \Omega$, the function spaces defined out of $\mathcal{F}$\ or $\mathcal{F}_\Omega$\ are identical, namely
\begin{itemize}
    \item $\mathcal{F}_c(V)=\mathcal{F}_{\Omega,\,c}(V)$, $\mathcal{F}_{\scaleto{\mbox{loc}}{5pt}}(V)=\mathcal{F}_{\Omega,\,\scaleto{\mbox{loc}}{5pt}}(V)$;
    \item for any $I\subset \mathbb{R}$, $\mathcal{F}_c(I\times V)=\mathcal{F}_{\Omega,\,c}(I\times V)$, $\mathcal{F}_{\scaleto{\mbox{loc}}{5pt}}(I\times V)=\mathcal{F}_{\Omega,\,\scaleto{\mbox{loc}}{5pt}}(I\times V)$;
    \item $\mathcal{F}_c(I\times V)\cap C_c^\infty(I\rightarrow \mathcal{F})=\mathcal{F}_{\Omega,\,c}(I\times V)\cap C_c^\infty(I\rightarrow \mathcal{F}_\Omega)$.
\end{itemize}
In particular, for any $f\in L^2_{\scaleto{\mbox{loc}}{5pt}}(I\times \Omega)$, the above discussion implies that the set of local weak solutions of $(\partial_t+P)u=f$\ on $I\times V$\ agrees with the set of local weak solutions of $(\partial_t+P_\Omega)u=f$\ on $I\times V$. Here $-P$\ and $-P_\Omega$\ stand for the generators of $(\mathcal{E},\mathcal{F})$\ and $(\mathcal{E}^\Omega_0,\mathcal{F}^\Omega_0)$, respectively.
\end{example}
\begin{example}
\label{unifellip}
Theorem \ref{mainthm} applied to the special case of Example \ref{restrictionformeg} gives exactly Theorem \ref{modelthm}. We illustrate it in the concrete example of bilinear forms associated with locally uniformly elliptic divergence form operators on $\mathbb{R}^n$. More precisely, let $m(x)dx$\ be a measure on $\mathbb{R}^n$\ with positive locally (essentially) bounded density $m(x)$, thus it is locally equivalent to the Lebesgue measure. Let $A:=\left( a_{ij}(x) \right)_{n\times n}$\ be a symmetric coefficient matrix with entries $a_{ij}(x)$\ being measurable locally (essentially) bounded functions on $\mathbb{R}^n$. Assume the matrix satisfies the local uniform ellipticity condition: for any $\Omega\Subset \mathbb{R}^n$, there exist some $0<c_\Omega<C_\Omega<\infty$, such that for a.e. $x\in \Omega$\ and any $\xi=(\xi_1,\cdots,\xi_n)\in \mathbb{R}^n$,
\begin{eqnarray}
\label{localuniformellipticity}
c_\Omega\sum_{i=1}^{n}\xi_i^2\leq \sum_{i,j=1}^{n}a_{ij}(x)\xi_i\xi_j\leq C_\Omega \sum_{i=1}^{n}\xi_i^2.
\end{eqnarray}
The bilinear form $\mathcal{B}_A$\ associated with $\left( a_{ij}(x) \right)_{n\times n}$\ is well-defined for smooth compactly supported functions as
\begin{eqnarray}
\mathcal{B}_A(f,g)=\int_{\mathbb{R}^n}\sum_{i,j=1}^n a_{ij}(x)\partial_{x_i}f(x)\partial_{x_j}g(x)\,m(x)dx.
\end{eqnarray}
By the local uniform ellipticity condition, it can be shown that $\mathcal{B}_A$\ is closable and the so-obtained Dirichlet form and its corresponding semigroup satisfy the conditions in Theorem \ref{mainthm} or Theorem \ref{modelthm}, cf. \cite{Ma,unifelliptic}. Consider the class of functions $\widetilde{C}$\ made up of functions $u$\ in $H_{\scaleto{\mbox{loc}}{5pt}}^1(I\times\mathbb{R}^n)$\ satisfying that for any $\varphi\in C_c^\infty(I\times\mathbb{R}^n)$,
\begin{eqnarray*}
\int_I\int_X(\partial_tu)\varphi\,dxdt+\int_I\mathcal{B}_A(u,\varphi)\,dt=0.
\end{eqnarray*}
Then these functions belong to $C^\infty(I\rightarrow C(\mathbb{R}^n))$. They are also locally in $W^{\infty,2}(I\rightarrow H^1(\mathbb{R}^n))\subset C^\infty(I\rightarrow H^1(\mathbb{R}^n))$\ by Theorem 4.1 in \cite{L2}.
%By Theorem \ref{mainthm}, these functions are locally in $C^\infty(I\rightarrow H_0^1(\Omega))\cap C^\infty(I\rightarrow C(\Omega))$\ for any $\Omega\Subset \mathbb{R}^n$.
\end{example}
\begin{example}
More generally, any Dirichlet form $(\mathcal{E},\mathcal{F})$\ on $L^2(X,m)$\ can be viewed as a bilinear triple with $\mathcal{B}=\mathcal{E}$, $\widetilde{\mathcal{D}}=\mathcal{F}_{\scaleto{\mbox{loc}}{5pt}}(X)$, and $\widetilde{\mathcal{D}}_0=\mathcal{F}_c(X)$. If on some $\Omega\Subset X$, $(\mathcal{E},\mathcal{F}_{\scaleto{\mbox{loc}}{5pt}}(X),\mathcal{F}_c(X))$\ is represented by a symmetric local regular Dirichlet form $(\mathcal{E}_\Omega,\mathcal{F}_\Omega)$\ that satisfies the conditions in Theorem \ref{mainthm}, then the theorem applies.

Note that when the Dirichlet form $(\mathcal{E},\mathcal{F})$\ is not local, there is not a good notion of local weak solutions for the associated heat equation $(\partial_t+P)u=0$\ in general. On the other hand, when the form agrees with a symmetric local regular Dirichlet form locally (i.e. on some $\Omega\Subset X$\ as in the sense of Definition \ref{agreelocally}), and when $(\mathcal{E}_\Omega,\mathcal{F}_\Omega)$\ satisfies Assumption \ref{cutoff} (existence of nice cut-off functions), one can check that for any $V\subset \Omega$, $I\subset \mathbb{R}$,
\begin{itemize}
    \item $\mathcal{F}_c(I\times V)=\mathcal{F}_{\Omega,\,c}(I\times V)$, $\mathcal{F}_{\scaleto{\mbox{loc}}{5pt}}(I\times V)=\mathcal{F}_{\Omega,\,\scaleto{\mbox{loc}}{5pt}}(I\times V)$;
    \item $\mathcal{F}_c(I\times V)\cap C_c^\infty(I\rightarrow \mathcal{F})=\mathcal{F}_{\Omega,\,c}(I\times V)\cap C_c^\infty(I\rightarrow \mathcal{F}_\Omega)$.
\end{itemize}
We verify the first item below. Hence there is a natural notion of local weak solutions of the heat equation $(\partial_t+P)u=0$\ on $I\times V$\ for any open subset $V\subset \Omega$\ (or with appropriate, nonzero right-hand side).

This type of examples includes 
\begin{itemize}
    \item symmetric local Dirichlet forms with any boundary condition (so that they may or may not be regular on the whole space);
    \item possibly nonlocal Dirichlet forms with jump measure $J(dxdy)$\ supported away from some open set $\Omega$, i.e., $J(dxdy)=1_{\Omega^c\times\Omega^c}J(dxdy)$. This condition is captured by the third item in Definition \ref{agreelocally}.
\end{itemize}
Proof of $\mathcal{F}_c(I\times V)=\mathcal{F}_{\Omega,\,c}(I\times V)$. 
\begin{itemize}
    \item $\mathcal{F}_c(I\times V)\subset\mathcal{F}_{\Omega,\,c}(I\times V)$: for any $v\in \mathcal{F}_c(I\times V)$, 
\begin{itemize}
    \item $v\in L^2(I\rightarrow \mathcal{F})$\ and $v$\ has compact support in $I\times V$. In particular, for a.e. $t\in I$, $v(t,\cdot)=v^t(\cdot)\in \mathcal{F}$, and $v^t|_\Omega=v^t$. Hence $v^t\in \mathcal{F}_{\Omega,\,\scaleto{\mbox{loc}}{5pt}}$. Moreover, $v^t$\ is compactly supported in $V\subset \Omega$, hence $v^t\in \mathcal{F}_{\Omega,\,c}$.
    \item By the previous item and (iii) in Definition \ref{agreelocally}, $\mathcal{E}(v^t,v^t)=\mathcal{E}_\Omega(v^t,v^t)$. Integrating over $t$\ gives $\int_I\mathcal{E}_\Omega(v^t,v^t)\,dt=\int_I\mathcal{E}(v^t,v^t)\,dt<+\infty$.
    \item To show $v\big|_{I\times \Omega}=v\in L^2(I\rightarrow \mathcal{F}_\Omega)$, it remains to check $v\in L^2(I\times \Omega, dtd\mu_\Omega)$. This follows from the assumption that $\mu_\Omega$\ and $m|_\Omega$\ are comparable.
\end{itemize}
\item $\mathcal{F}_{\Omega,\,c}(I\times V)\subset \mathcal{F}_c(I\times V)$: for any $v\in \mathcal{F}_{\Omega,\,c}(I\times V)$
\begin{itemize}
    \item For a.e. $t\in I$, $v^t\in \mathcal{F}_{\Omega,\,c}\subset \mathcal{F}$. Hence $\mathcal{E}_\Omega(v^t,v^t)=\mathcal{E}(v^t,v^t)$\ and so is the integral in time.
\end{itemize}
\end{itemize}
Proof of $\mathcal{F}_{\scaleto{\mbox{loc}}{5pt}}(I\times V)=\mathcal{F}_{\Omega,\,\scaleto{\mbox{loc}}{5pt}}(I\times V)$.
\begin{itemize}
    \item For any $v\in \mathcal{F}_{\Omega,\,\scaleto{\mbox{loc}}{5pt}}(I\times V)$, for any $J\times W\Subset I\times V$, there exists some $v^\sharp\in L^2(I\rightarrow \mathcal{F}_\Omega)$, such that $v^\sharp=v$\ a.e. on $J\times W$. By multiplying with a nice product cut-off function if necessary, we may assume $v^\sharp$\ is compactly supported in $I\times V$. Since $\mathcal{F}_c(I\times V)=\mathcal{F}_{\Omega,\,c}(I\times V)$, in particular, $v^\sharp$\ belongs to $L^2(I\rightarrow \mathcal{F})$\ as well. So $v\in \mathcal{F}_{\scaleto{\mbox{loc}}{5pt}}(I\times V)$.
    \item For any $v\in \mathcal{F}_{\scaleto{\mbox{loc}}{5pt}}(I\times V)$, for any $J\times W\Subset I\times V$, there exists some $v^\sharp\in L^2(I\rightarrow \mathcal{F})$, such that $v^\sharp=v$\ a.e. on $J\times W$. 
    %Without loss of generality, we can assume further that $v^\sharp\in C^\infty_c(I\rightarrow \mathcal{F})$. 
    If for any nice cut-off function $\eta\in \mathcal{F}_{\Omega,\,c}(V)$, $\eta(x)v^\sharp(t,x)\in L^2(I\rightarrow \mathcal{F}_\Omega)$, then since $\eta$\ can be chosen properly so that $\eta v^\sharp=v$\ a.e. on $J\times W$, $v\in \mathcal{F}_{\Omega,\,\scaleto{\mbox{loc}}{5pt}}(I\times V)$. We now show that $\eta(x)v^\sharp(t,x)\in L^2(I\rightarrow \mathcal{F}_\Omega)$.
    \begin{itemize}
        \item Note that for a.e. $t\in I$, $v^\sharp(t,\cdot)\big|_\Omega\in \mathcal{F}_{\Omega,\,\scaleto{\mbox{loc}}{5pt}}$\ by (i) in Definition \ref{agreelocally}, so $\eta(\cdot)v^\sharp(t,\cdot)\in \mathcal{F}_{\Omega,\,c}$. Denote $v^\sharp(t,\cdot)$\ by $v^{\sharp,\,t}(\cdot)$, then $\mathcal{E}_\Omega(\eta\,v^{\sharp,\,t},\,\eta\,v^{\sharp,\,t})=\mathcal{E}(\eta\,v^{\sharp,\,t},\,\eta\,v^{\sharp,\,t})$.
        \item If $v^\sharp$\ is of product form $v^\sharp(t,x)=f(t)g(x)$, then clearly $\eta(x)v^\sharp(t,x)\in L^2(I\rightarrow\mathcal{F})$, and integrating in $t$\ the conclusion in the previous item shows that $\eta v^\sharp\in L^2(I\rightarrow \mathcal{F}_\Omega)$. For general $v^\sharp$, by approximation using simple functions in $L^2(I\rightarrow \mathcal{F})$, it follows that $\eta v^\sharp\in L^2(I\rightarrow \mathcal{F})$, and hence in $L^2(I\rightarrow \mathcal{F}_\Omega)$.
        %\item Let $v_n:=(v^\sharp\vee (-n))\wedge n$, then $v_n\in \mathcal{F}\subset \mathcal{F}_{\Omega,\,\scaleto{\mbox{loc}}{5pt}}$, and 
        %\begin{eqnarray*}
        %\lefteqn{\left(\mathcal{E}_\Omega(\eta v_n^t,\,\eta v_n^t)\right)^{1/2}=\left(\mathcal{E}(\eta v_n^t,\,\eta v_n^t)\right)^{1/2}}\\
        %&\leq& n\left(\mathcal{E}(\eta,\eta)\right)^{1/2}+\left(\mathcal{E}(v_n^t,\,v_n^t)\right)^{1/2}\leq nC+\left(\mathcal{E}_1(v^{\sharp,\,t},\,v^{\sharp,\,t})\right)^{1/2}.
        %\end{eqnarray*}
        %Integrating in $t$\ shows that $\eta v_n\in L^2(I\rightarrow \mathcal{F}_\Omega)$.
        %\item Let $\gamma_n(t):=\mathcal{E}(\eta v_n^t,\, \eta v_n^t)\geq 0$. Then $\gamma_n(t)\leq \gamma_{n+1}(t)$.
        %\item $\mathcal{E}(\eta v_n^t,\,\eta v_n^t)\nearrow \mathcal{E}(\eta v^{\sharp,\,t},\,\eta v^{\sharp,\,t})$, since each is a normal contraction of the next. So
        %\begin{eqnarray*}
        %\lim_{n\rightarrow \infty}\int_I\mathcal{E}(\eta v_n^t,\,\eta v_n^t)\,dt=\int_I\mathcal{E}(\eta v^{\sharp,\,t},\,\eta v^{\sharp,\,t})\,dt
        %\end{eqnarray*}
        \end{itemize}
    %$$\mathcal{E}_\Omega(\eta\,v^{\sharp,\,t},\,\eta\,v^{\sharp,\,t})=\mathcal{E}(\eta\,v^{\sharp,\,t},\,\eta\,v^{\sharp,\,t})\leq \left(\mathcal{E}_1(\eta,\eta)\right)^{1/2}\left(\mathcal{E}_1(v^{\sharp,\,t},\,v^{\sharp,\,t})\right)^{1/2}.$$
    %The inequality follows from that $(\mathcal{E},\mathcal{F})$\ is a Dirichlet form. Integrating in time then proves that $\left|\left|\eta v^\sharp\right|\right|_{L^2(I\rightarrow \mathcal{F}_\Omega)}<\infty$.
    %which is uniformly bounded in $t\in I$.
    %$$\mathcal{E}_\Omega(\eta\,\partial_tv^{\sharp,\,t},\,\eta\,\partial_tv^{\sharp,\,t})=\mathcal{E}_\Omega(\eta,\eta \,\partial_tv^\sharp)+\mathcal{E}_\Omega(\partial_tv^\sharp,\eta\partial_tv^\sharp)=\mathcal{E}(\eta,\eta \,\partial_tv^\sharp)+\mathcal{E}(\partial_tv^\sharp,\eta\partial_tv^\sharp)$$
\end{itemize}
\end{example}

\subsection{Other examples}
%The application of Theorem \ref{mainthm0} or Theorem \ref{mainthm} can be further extended to settings beyond where there is a bilinear form that is locally represented by a symmetric local regular Dirichlet form. Assume there is a set of functions $\mathcal{A}$, usually ones that are solutions of some partial differential equation in some sense, and assume there exists a symmetric local regular Dirichlet form $(\mathcal{E},\mathcal{F})$\ such that this set of functions are local weak solutions of the heat equation associated with the Dirichlet form $(\mathcal{E},\mathcal{F})$. Here $(\mathcal{E},\mathcal{F})$\ is defined on the $L^2$\ space of an open subset of the whole metric measure space. Then Theorem \ref{mainthm0} applies to the set of functions $\mathcal{A}$, if $(\mathcal{E},\mathcal{F})$\ and its corresponding semigroup satisfy the conditions in Theorem \ref{mainthm0}. We give two examples of this type.
We now step out of the global bilinear form setting and give two more examples.
\begin{example}
As in Example \ref{unifellip}, let $A=\left( a_{ij}(x) \right)_{n\times n}$\ be symmetric, measurable, locally bounded, and locally uniformly elliptic. Let $\Omega$\ be a domain in $\mathbb{R}^n$\ that may or may not be simply-connected. Let ${\bf b}(x):\Omega\rightarrow \mathbb{R}^n$\ be a locally bounded vector field that is given locally by ${\bf b}(x)=\nabla \log{\mu}(x)$\ for some differentiable function $\mu$. That is, for any $U\Subset \Omega$, ${\bf b}(x)=\nabla \log{\mu_{\scaleto{U}{4pt}}}(x)$\ for some differentiable function $\mu_U$\ on $V$. Consider the second order differential operator $L$\ defined by
\begin{eqnarray*}
Lu:=\mbox{div}{\left(A(x)\nabla u\right)}+{\bf b}(x)\cdot A(x)\nabla u,
\end{eqnarray*}
for any compactly supported smooth function $u$. The precise interpretation of $L$\ is through the (nonsymmetric) bilinear form associated with it. This example is considered in for example \cite{drifteg}. On each small open ball $U=B(x;r)$\ where ${\bf b}(x)=\nabla \log{\mu_{\scaleto{U}{4pt}}}(x)$, write $\mu_{\scaleto{U}{4pt}}(x)=:\mu(x)$, and define another differential operator $\widetilde{L}$\ by
\begin{eqnarray*}
\widetilde{L}u=\frac{1}{\mu(x)}\mbox{div}{(\mu(x)A(x)\nabla u)}.
\end{eqnarray*}
%If ${\bf b}(x)=\nabla \log{\mu}(x)$\ holds on $\Omega$, 
Then $\widetilde{L}$\ is formally symmetric on $L^2(U,\,\mu(x)dx)$, namely for smooth functions $u,v$\ with supports in $U$,
\begin{eqnarray*}
\int_{U}(\widetilde{L}u)v\,\mu(x)dx=\int_{U}-\nabla v\cdot A(x)\nabla u\,\mu(x)dx;
\end{eqnarray*}
$L$\ and $\widetilde{L}$\ coincide formally in the sense that
\begin{eqnarray*}
\int_{U}(Lu)v\,dx=\int_{U}(\widetilde{L}u)v\,dx.
\end{eqnarray*}
%Since $\Omega$\ might not be simply-connected, the equation ${\bf b}(x)=\nabla \log{\mu}(x)$\ might only have solutions locally. 
Let $(\mathcal{E}_U,\mathcal{F}_U)$\ be the Dirichlet form on $L^2(U,\mu(x)dx)$\ associated with $\widetilde{L}$\ with the minimal domain $\mathcal{F}_U$, i.e., $\mathcal{F}_U$\ is obtained by completion of $C_c^\infty(U)$\ w.r.t. the $\mathcal{E}_{U,1}$\ norm. Since $\mu(x)$\ is locally bounded, $\mathcal{F}_U=H^1_0(U,\mu(x)dx)=H^1_0(U,dx)$. $(\mathcal{E}_U,\mathcal{F}_U)$\ is symmetric strongly local regular, and satisfies Assumption \ref{cutoff} (existence of nice cut-off functions). Its corresponding semigroup admits a continuous kernel which satisfies the Gaussian bound, cf. \cite{drifteg,unifelliptic}. For any open interval $I\subset \mathbb{R}$, any local weak solution $u$\ of $(\partial_t+L)u=f$\ on $I\times U$\ in the sense that $u$\ is locally in $L^2(I\rightarrow H^1(U))$\ and satisfies for any $\varphi\in C_c^\infty(I\times U)$,
\begin{eqnarray*}
\lefteqn{\hspace{-.1in}-\int_I\int_U u\partial_t\varphi\,dxdt-\int_I\int_U A(x)\nabla_x u\cdot \nabla_x \varphi\,dxdt+\int_I\int_U\varphi{\bf b}(x)\cdot A(x)\nabla_x u\,dxdt}\\
&&\hspace{3.3in}=\int_I\int_\Omega f\varphi\,dxdt,
\end{eqnarray*}
is a local weak solution of $(\partial_t+\widetilde{L})u=f$\ on $I\times U$. Here $f\in L^2_{\scaleto{\mbox{loc}}{5pt}}(U,dx)$. Indeed, for any $\phi\in C_c^\infty(I\times U)$, the above equality applied to $\varphi(t,x)=\phi(t,x)\mu(x)$\ gives (recall that ${\bf b}(x)=\nabla \log{\mu}(x)$\ on $U$)
\begin{eqnarray*}
-\int_I\int_U u\partial_t\phi\,\mu(x)dxdt-\int_I\int_U A(x)\nabla_x u\cdot \nabla_x \phi\,\mu(x)dxdt=\int_I\int_\Omega f\phi\,\mu(x)dxdt.
\end{eqnarray*}
Therefore, Theorem \ref{generalthm} applies, $u$\ is continuous on $I\times U$\ and hence on $I\times \Omega$.
\end{example}

%In Example \ref{unifellip}, there is a natural symmetric bilinear form $\mathcal{B}_A$, 
%although globally it is not necessarily a Dirichlet form. 
%while when the coefficients $a_{ij}(x)$\ are only measurable functions, it is hard to explicitly interpret the generator ``$\frac{1}{m(x)}\sum \partial_{x_j}(m(x)a_{ij}(x)\partial_{x_i})=\frac{1}{m(x)}\mbox{div}{(m(x)A(x)\nabla)} $''. 

%There are examples of the opposite type, where the differential operator is explicit, but it is not easy or not possible to define a global Dirichlet form corresponding to it. For example, given a locally uniformly elliptic coefficient matrix $\left(a_{ij}(x)\right)$\ as in Example \ref{unifellip}, we consider the differential operator $L:=\sum a_{ij}(x)\partial_{x_i}\partial_{x_j}$. Since it is not clear if $L^*u=0$\ has a global solution on $\mathbb{R}^n$, it is not clear if there is an invariant measure for $L$, and hence no associated Dirichlet form is guaranteed. On the other hand, when we look at each open ball with finite radius, $L^*u=0$\ is known to have positive solutions, and we can talk about Dirichlet forms. In particular, any strong solution to $(\partial_t-L)u=0$\ on $I\times \Omega$, where $I\subset \mathbb{R}$\ and $\Omega$\ is any open ball with finite radius, is a local weak solution on $I\times \Omega$. And Theorem \ref{mainthm} thus indicates that the solutions as well as any of their time derivatives are locally bounded in $I\times\Omega$.
\begin{example}[Caloric functions on graph]
Let $X=(V,E)$\ be a graph with vertex set $V$\ and edge set $E$. Assume that $V,E$\ are countable sets and there is a map $E\rightarrow V\times V$\ with $e\mapsto (e_-,e_+)$, where $e_-\neq e_+$. There is no loop, but multiple edges could be between two vertices; each edge is oriented. For simplicity, assume each edge has length $1$, i.e., each edge can be identified with $(0,1)$. For each vertex $v\in V$, let $E_v$\ denote the set of edges having $v$\ as a vertex. Assume that each $|E_v|$\ is finite. Let ${\bf b}$\ be a function that assigns to each $v\in V$\ a weight map ${\bf b}(v):E_v\rightarrow \mathbb{R}_+,\ e\mapsto b_v(e)$. For each $e\in E$, denote the restriction of any function $f$\ on $e$\ by $f_e$. Because each edge $e$\ is oriented, the derivative $f_e'$\ is well-defined.  For any vertex $v$\ and any $e\in E_v$, let $\epsilon_v(e)=1$\ if $v=e_+$, and $\epsilon_v(e)=-1$\ if $v=e_-$. With this notation, $\epsilon_v(e)f_e'(v)$\ represents the outward normal derivative of $f$\ at vertex $v$, along the edge $e$. We define a collection of function spaces recursively as follows. The stated conditions on edges and vertices should hold for all edges and vertices.
\begin{itemize}
    \item $\mathcal{C}_{b}^0$\ consists of all continuous functions on $X$;
    \item $\mathcal{C}_{b}^1:=\left\{f\in \mathcal{C}_{b}^0:f_e\in C^1(\overline{e}),\,\sum\limits_{e\in E_v} \epsilon_v(e)b_v(e)f'_e(v)=0\right\}$;
    \item $\mathcal{C}_{b}^2:=\left\{f\in \mathcal{C}_{b}^1:f_e\in C^2(\overline{e}),\,f''_{e}(v)=f''_{\tilde{e}}(v)\ \mbox{for any $e,\tilde{e}\in E_v$}\right\}$;
    \item $\mathcal{C}_b^{2k+1}:=\left\{f\in \mathcal{C}_b^{2k}: f_e\in C^{2k+1}(\overline{e}),\,\sum\limits_{e\in E_v} \epsilon_v(e)b_v(e)f^{(2k+1)}_e(v)=0\right\}$;
    \item $\mathcal{C}_b^{2k+2}:=\\[0.05in]
    \left\{f\in \mathcal{C}_b^{2k+1}: f_e\in C^{2k+2}(\overline{e}),\,f^{(2k+2)}_{e}(v)=f^{(2k+2)}_{\tilde{e}}(v)\ \mbox{for any $e,\tilde{e}\in E_v$}\right\}$.
\end{itemize}
The conditions $\sum\limits_{e\in E_v} \epsilon_v(e)b_v(e)f'_e(v)=0$, $\sum\limits_{e\in E_v} \epsilon_v(e)b_v(e)f^{(2k+1)}_e(v)=0$\ are often called the Kirchhoff condition.
%Consider an associated Laplace operator $\Delta_{\bf b}$\ defined on $C^\infty(X)$. Then for any $u\in C^\infty(X)$, $\Delta_{\bf b}u=0$\ implies
%\begin{itemize}
%	\item $u''=0$\ in the interior of each edge;
%	\item at each vertex $v\in V$,
%	\begin{eqnarray}
%	\label{graph1}
%	\sum_{e\in E_v}b_v(e)\,\partial_\nu u=0,
%	\end{eqnarray}
%	where $\partial_\nu u$\ represents the outward normal derivative of $u$.
%\end{itemize}
For any open subset $\Omega\subset X$\ and open interval $I\subset \mathbb{R}$, for any $n\in \mathbb{N}$, let $C_b^{n}(\Omega)$\ be defined as above with conditions restricted to edges and vertices in $\Omega$; consider the function space $C^1(I\rightarrow \mathcal{C}_b^{2*}(\Omega))$, where
\begin{eqnarray*}
C_b^{2*}(\Omega):=\left\{f\in C_b^1(\Omega):f_e\in C^2(\overline{e}),\,\mbox{for any $e$\ with $e\cap \Omega\neq \emptyset$}\right\},
\end{eqnarray*}
and the subspace of $C^1(I\rightarrow \mathcal{C}_b^{2*}(\Omega))$\ in which functions further satisfy the heat equation $\partial_tu_e-\partial_{xx}u_e=0$\ on each edge $e$\ with $e\cap\Omega\neq \emptyset$. We denote this subspace by $\mathfrak{C}(I\times\Omega)$\ and call functions in this subspace {\em caloric functions} on $I\times\Omega$. Note that by definition, for any fixed $t\in I$, caloric functions $u$\ satisfy the Kirchhoff condition
\begin{eqnarray}
\label{kirch}
\sum\limits_{e\in E_v} \epsilon_v(e)b_v(e)\partial_xu_e(v)=0.
\end{eqnarray}

%When the edge weights are symmetric (i.e. $b_{e_+}(e)=b_{e_-}(e)$) or when $\Omega$\ has a tree structure, 

For any vertex $v$, consider any collection $\mathcal{S}(v)$\ of vertices in $X$\ that contains $v$, such that these vertices, together with all edges in $\{E_w,\,w\in \mathcal{S}(v)\}$, is an $\mathbb{R}$-tree. Denote this $\mathbb{R}$-tree by $T_\mathcal{S}(v)$. In Figure \ref{figure} we give two examples of $T_\mathcal{S}(v)$\ inside some grid as the whole space. The hollow nodes stand for vertices not included in the $\mathbb{R}$-tree $T_\mathcal{S}(v)$. The first one is the simplest, containing the vertex $v$\ and all four edges attached to it. The second example contains more vertices, all marked as filled nodes, together with all edges attached to these vertices. Recall that we consider open edges.

\begin{center}
\begin{tikzpicture}
\draw [thick] (0,0) -- (1,0);
\draw [thick] (0,0) -- (-1,0);
\draw [thick] (0,0) -- (0,1);
\draw [thick] (0,0) -- (0,-1);
\draw [fill] (0,0) circle [radius=0.1]; 
\node at (.15,-.2) {$v$};
\draw (1,0) circle [radius=0.1];
\draw (-1,0) circle [radius=0.1];
\draw (0,1) circle [radius=0.1];
\draw (0,-1) circle [radius=0.1];
\end{tikzpicture}
\hspace{0.5in}
\begin{tikzpicture}
\draw[step=1cm,black,thick] (-1,-1) grid (1,2);
\draw [fill] (0,0) circle [radius=0.1]; 
\node at (.15,-.2)  {$v$};
%\draw [fill] (0,0) circle [radius=0.1];
\draw [fill] (1,0) circle [radius=0.1];
\draw [fill] (-1,0) circle [radius=0.1];
\draw [fill] (0,-1) circle [radius=0.1];
\draw [fill] (0,1) circle [radius=0.1];
\draw [fill] (0,2) circle [radius=0.1];
\draw [fill] (-1,2) circle [radius=0.1];
\draw [fill] (1,2) circle [radius=0.1];
\draw [fill] (-1,3) circle [radius=0.1];
\draw (-1,4) circle [radius=0.1];
\draw (0,3) circle [radius=0.1];
\draw (1,3) circle [radius=0.1];
\draw (-2,3) circle [radius=0.1];
\draw (-1,1) circle [radius=0.1];
\draw (1,1) circle [radius=0.1];
\draw (-1,-1) circle [radius=0.1];
\draw (1,-1) circle [radius=0.1];
\draw (0,-2) circle [radius=0.1];
\draw (2,0) circle [radius=0.1];
\draw (-2,2) circle [radius=0.1];
\draw (2,2) circle [radius=0.1];
\draw (-2,0) circle [radius=0.1];
\draw [thick] (1,0) -- (2,0);
\draw [thick] (1,2) -- (2,2);
\draw [thick] (1,2) -- (1,3);
\draw [thick] (0,2) -- (0,3);
\draw [thick] (-1,2) -- (-1,3);
\draw [thick] (-1,2) -- (-2,2);
\draw [thick] (-1,3) -- (-2,3);
\draw [thick] (-1,3) -- (0,3);
\draw [thick] (-1,3) -- (-1,4);
\draw [thick] (-1,0) -- (-2,0);
\draw [thick] (0,-1) -- (0,-2);
\end{tikzpicture}
\captionof{figure}{Examples of admissible $T_\mathcal{S}(v)$}\label{figure}
\end{center}

Note that at each vertex $w$, for the Kirchhoff condition (\ref{kirch}) to hold, only the ratios among the edge weights $b_w(e)$, $e\in E_w$\ matter. Therefore, by starting with edges in $E_v$\ and recursively modifying the edge weights for vertices moving away from $v$\ in $T_\mathcal{S}(v)$, we can obtain a new weight function $\widetilde{b}$\ satisfying
\begin{itemize}
   \item[(1)] at each $w\in T_\mathcal{S}(v)$, the ratios among the new edge weights for edges in $E_w$\ are the same as the old ones. That is, for any $e,\widetilde{e}\in E_w$, $\frac{b_w(e)}{b_w(\widetilde{e})}=\frac{\widetilde{b}_w(e)}{\widetilde{b}_w(\widetilde{e})}$. This condition implies that $C_b^n(\Omega)=C_{\widetilde{b}}^n(\Omega)$\ for all $n\in \mathbb{N}$.
   \item[(2)] $\widetilde{b}$\ is symmetric in $T_\mathcal{S}(v)$, i.e., $\widetilde{b}_{e_-}(e)=\widetilde{b}_{e_+}(e)$\ for all edges $e$\ and vertices $e_-,e_+$\ in $T_\mathcal{S}(v)$. In particular, we may write without subscript as $\widetilde{b}(e)$.
\end{itemize}
Note that such $\widetilde{b}$\ can be obtained because $T_\mathcal{S}(v)$\ is a tree.

To simplify notation, let $\Omega:=T_\mathcal{S}(v)$. Consider the Dirichlet form $(\mathcal{E}^\Omega_0,\mathcal{F}^\Omega_0)$\ defined on $\Omega$\ as follows. Let the measure $\mu_\Omega$\ on $\Omega$\ be such that its restriction on each edge $e\subset \Omega$\ equals $\widetilde{b}(e)dx$, where $dx$\ denotes the standard Lebesgue measure on that edge. The domain of the Dirichlet form, $\mathcal{F}^\Omega_0$, contains functions $f$\ that belong to $H^1_0(\Omega,\mu_\Omega)$. The subscript $0$\ stands for Dirichlet boundary condition. Note that $H^1_0(\Omega,\mu_\Omega)\subset C(\Omega)$. For any $f,g\in \mathcal{F}^\Omega_0$, $\mathcal{E}^\Omega_0(f,g)$\ is defined as
\begin{eqnarray*}
\mathcal{E}^\Omega_0(f,g)=\sum_{e}\int_{e\cap \Omega}f_e'(x) g_e'(x)\,\widetilde{b}(e)dx.
\end{eqnarray*}
$(\mathcal{E}^\Omega_0,\mathcal{F}^\Omega_0)$\ is a symmetric (strongly) local regular Dirichlet form. Let $-\Delta_{\widetilde{b}}^\Omega$\ denote the generator of $(\mathcal{E}^\Omega_0,\mathcal{F}^\Omega_0)$. For any function $f$\ in the domain of $-\Delta_{\widetilde{b}}^\Omega$, on any edge $e$\ in $\Omega$, $\Delta_{\widetilde{b}}^\Omega f_e=-f''_e$.

Integration by parts shows that any function $u$\ in $\mathfrak{C}(I\times\Omega)$\ is a local weak solution of the heat equation $\partial_tu+\Delta_{\widetilde{b}}^\Omega u=0$\ on $I\times \Omega$: $u\in H^1_{\scaleto{\mbox{loc}}{5pt}}(I\times\Omega)\subset \mathcal{F}_{\scaleto{\mbox{loc}}{5pt}}(I\times\Omega)$; for any $\varphi\in C^\infty(I\rightarrow \cap_{n\in \mathbb{N}}C_b^{n}(\Omega))$\ with compact support in $I\times \Omega$, 
\begin{eqnarray*}
-\int_I\int_\Omega u\partial_t\varphi\,d\mu_\Omega dt-\int_I\mathcal{E}^\Omega_0(u,\varphi)\,dt=0.
\end{eqnarray*}
In \cite{Riemcplx} it is shown that $(\mathcal{E}^\Omega_0,\mathcal{F}^\Omega_0)$\ satisfies the assumptions in Theorem \ref{generalthm}. Its corresponding semigroup admits a continuous kernel. The following proposition follows from Theorem \ref{generalthm}.
\begin{proposition}
In the above setting, let $I\subset \mathbb{R}$\ be an open interval and $\Omega=T_\mathcal{S}(v)$\ for some vertex $v$\ and selection of vertices $\mathcal{S}(v)$. Then all caloric functions on $I\times \Omega$\ belong to $C^\infty\left(I\rightarrow \cap_{n\in\mathbb{N}}C_b^{n}(\Omega)\right)$, and all time derivatives of the caloric functions are still caloric functions.
\end{proposition}
%for any $I\subset \mathbb{R}$, any local weak solution of $(\partial_t+\Delta_{\bf b}^\Omega)u=0$\ on $I\times \Omega$\ satisfies that for any $n\in \mathbb{N}$, $\partial_t^nu$\ is continuous on $\Omega$, and $\partial_t^nu$ is a local weak solution of $(\partial_t+\Delta_{\bf b}^\Omega)\partial_t^nu=0$\ on $I\times \Omega$.  
\begin{proof}Let $u$\ be any caloric function on $I\times\Omega$. On each edge $e$\ in $\Omega$, by definition of caloric functions, $u$\ satisfies $\partial_tu_e-\partial_{xx}u_e=0$\ and hence is smooth in time and space. As discussed above, by Theorem \ref{generalthm}, $u$\ is locally in $C^{\infty}(I\rightarrow \mathcal{F}^\Omega_0)$; all time derivatives of $u$\ are local weak solutions of the heat equation, i.e, for any $n\in \mathbb{N}$, $\partial_t^nu$\ is a local weak solution of $\partial_t(\partial_t^nu)+\Delta^\Omega_{\widetilde{b}}(\partial_t^nu)=0$\ on $I\times\Omega$; they are continuous on $I\times\Omega$.
It then follows that
\begin{itemize}
	\item[(1)] for any $m,n\in \mathbb{N}$, $\partial_t^n(\Delta_{\widetilde{b}}^\Omega)^mu$\ is continuous on $I\times\Omega$, i.e., the functions $\partial_t^n\partial_x^{2m}u_e$\ on each $I\times e$\ can be continuously extended to $I\times\Omega$;
	\item[(2)] by (1), for any fixed time $t$, any $n\in \mathbb{N}$, $\partial_t^nu_e(t,\cdot)\in C^\infty(\overline{e})$\ for all $e$\ in $\Omega$; for any $m,n\in \mathbb{N}$, at any vertex $w$\ in $\Omega$, 
\begin{eqnarray*}
	\sum_{e\in E_w}\epsilon_w(e)b_w(e)\,\partial_t^n\partial_x^{2m+1}u_e(w)=0.
\end{eqnarray*}
\end{itemize}
This completes the proof that $u\in C^\infty\left(I\rightarrow \cap_{n\in\mathbb{N}}C_b^{n}(\Omega)\right)$\ and all time derivatives of $u$\ are still caloric functions on $I\times\Omega$.
\end{proof}
Note that in general, there are no global Dirichlet form structures on $X$\ that correspond to heat equations admitting such caloric functions as local weak solutions.
\end{example}
\subsection{Proof of Theorem \ref{mainthm}}
Under the hypotheses, $u$\ (more precisely, $u\big|_{I\times U}$) is a local weak solution of $(\partial_t+P_\Omega)u=f$, namely
\begin{itemize}
    \item $u\big|_{I\times U}\in \mathcal{F}_{\Omega,\,\scaleto{\mbox{loc}}{5pt}}(I\times U)$;
    \item for any $\varphi\in C_c^\infty(I\rightarrow \mathcal{F}_{\Omega})\cap \mathcal{F}_{\Omega,\,c}(I\times U)$,
\begin{eqnarray*}
-\int_I\int_U u\cdot \partial_t\varphi\,d\mu_\Omega dt+\int_I\mathcal{E}_\Omega(u,\varphi)\,dt=\int_I\int_U f\varphi\,d\mu_\Omega dt.
\end{eqnarray*}
\end{itemize} 
We first show that without loss of generality, $(H_t^\Omega)_{t>0}$\ can be assumed to satisfy the global ultracontractivity condition. More precisely, for any $V\Subset \Omega$, the restriction form $(\mathcal{E}_{\Omega,\,V},\mathcal{F}_{\Omega,\,V})$\ of $(\mathcal{E}_\Omega,\mathcal{F}_\Omega)$\ on $V$\ (see Example \ref{restrictionformeg}, using notations there $(\mathcal{E}_{\Omega,\,V},\mathcal{F}_{\Omega,\,V})$\ should be written as $((\mathcal{E}_{\Omega})^{V}_0,(\mathcal{F}_{\Omega})^V_0)$) satisfies that
\begin{itemize}
    \item the bilinear triple $(\mathcal{B},\widetilde{\mathcal{D}},\widetilde{\mathcal{D}}_0)$\ is represented by $(\mathcal{E}_{\Omega,\,V},\mathcal{F}_{\Omega,\,V})$\ on $V$;
    \item the semigroup associated with $(\mathcal{E}_{\Omega,\,V},\mathcal{F}_{\Omega,\,V})$, denoted by $(H_t^{\Omega,\,V})_{t>0}$, satisfies the global ultracontractivity condition
    \begin{eqnarray*}
    ||H_t^{\Omega,\,V}||_{L^2(V)\rightarrow L^\infty(V)}\leq e^{M_V(t)};
    \end{eqnarray*}
    \item $(H_t^{\Omega,\,V})_{t>0}$\ also satisfies the $L^\infty$\ off-diagonal upper bound (\ref{globalGaussian})(\ref{mcontrolpoly});
    \item if $V\subset U$, $u\big|_{I\times V}$\ is a local weak solution of $(\partial_t+P_{\Omega,\,V})u=f$\ on $I\times V$. Here $-P_{\Omega,\,V}$\ is the generator of the restriction form $(\mathcal{E}_{\Omega,\,V},\mathcal{F}_{\Omega,\,V})$.
\end{itemize}    
Items 2 and 3 follow from the facts that for any $f\geq 0$\ in $L^2(V)$\ and supported in $V$, for any $t>0$\ and a.e. $x$\ in $V$,
\begin{eqnarray*}
H_t^{\Omega,\,V}f(x)\leq H^\Omega_tf(x),
\end{eqnarray*}
and that the semigroups are positivity preserving. Verifying the off-diagonal upper bounds for the time derivatives, i.e., for the terms $<\partial_t^kH_t^{\Omega,\,V}f,\,g>$\ with $k>0$, is similar to the proof of Lemma 8.4 in \cite{L2}. Hence, by passing to some restriction form if necessary, we assume that $(H_t^\Omega)_{t>0}$\ is globally ultracontractive. By Theorem \ref{mainthm0}, $u$\ locally belongs to $W^{n,\infty}(I\rightarrow L^\infty(U,\mu_\Omega))$.

Now suppose $H_t^\Omega$\ admits a density kernel $h_\Omega(t,x,y)$\ that is continuous on $(0,c)\times V\times V$\ for some $V\subset U$. To show continuity of $u$, consider the semigroup $H_t^{\Omega,\,W}$\ corresponding to the restriction form $(\mathcal{E}_{\Omega,\,W},\mathcal{F}_{\Omega,\,W})$, where $W\Subset V$\ is any precompact open subset. First note that the global ultracontractivity condition for $H_t^{\Omega,\,W}$\ guarantees the existence of an (essentially bounded) heat kernel $h_{\Omega,\,W}(t,x,y)$. By applying Corollary \ref{continuitythm} to the PDE $(\partial_t+P_{\Omega,\,W})u=f$\ on $I\times W$, it is enough to show that $h_\Omega(t,x,y)$\ being continuous on $(0,c)\times W\times W$\ implies $h_{\Omega,\,W}(t,x,y)$\ being continuous on $(0,c)\times W\times W$. Let $W'\Subset W$\ be any precompact open subset, we show that $h_{\Omega,\,W}$\ is continuous on $(0,c)\times W'\times W'$. By the Dynkin formula (cf. \cite{Dynkinformula}), the two kernels are related by
	\begin{eqnarray*}
		h_{\Omega,\,W}(t,x,y)=h_\Omega(t,x,y)-\int_{0}^{t}\int_{\partial W}h_\Omega(t-t',z,y)\,\mu_x(dt'dz),
	\end{eqnarray*}
	where $\partial W$\ denotes the boundary of $W$, and for each $x$, $\mu_x$\ is a probability measure on $[0,+\infty]\times\partial W$. For any $\epsilon>0$, by the $L^\infty$\ off-diagonal upper bound for the semigroup $H^\Omega_t$, since $W'$\ and $\partial W$\ are separated by disjoint open sets, there exists some $\delta(\epsilon)>0$\ such that
	\begin{eqnarray}
	\sup_{0<s<\delta(\epsilon)}\sup_{y\in W',\, z\in \partial W}h_\Omega(s,z,y)<\epsilon.
	\end{eqnarray}
	Here taking supremum is valid since $h_\Omega$\ is continuous. For any fixed $t\in (0,c)$, let $\delta=\min\{\frac{t}{2},\, \delta(\epsilon)\}$. Since $\mu_x$\ is a probability measure, the integral
	\begin{eqnarray}
	\int_{t-\delta}^{t}\int_{\partial W}h_\Omega(t-t',z,y)\,\mu_x(dt'dz)<\epsilon
	\end{eqnarray}
	for any $x,y\in W'$.
	
	For the same $\epsilon$, since $h_\Omega(t,z,y)$\ is continuous on $(0,c)\times V\times V$\ and $(\delta,t)\times W\times W\Subset (0,c)\times V\times V$, for any $x\in W'$\ and any $y_1,y_2\in W'$,
	\begin{eqnarray*}
		&&\hspace{-.2in}\left|\int_{0}^{t-\delta}\int_{\partial W}h_\Omega(t-t',z,y_1)\,\mu_x(dt'dz)-\int_{0}^{t-\delta}\int_{\partial W}h_\Omega(t-t',z,y_2)\,\mu_x(dt'dz)\right|\\
		&\leq& \sup_{r\in (\delta,\,t),\,z\in \partial W}|h_\Omega(r,z,y_1)-h_\Omega(r,z,y_2)|\int_{0}^{t-\delta}\int_{\partial W}\mu_x(dt'dz)\\
		&\leq& \sup_{r\in (\delta,\,t),\,z\in \partial W}|h_\Omega(r,z,y_1)-h_\Omega(r,z,y_2)|.
	\end{eqnarray*}
	This upper bound is independent of $x$\ and can be made less than $\epsilon$\ by taking $y_1$\ and $y_2$\ close. Hence for any $\epsilon>0$, for any $t\in (0,c)$, there exists some $d_0>0$\ such that for any $x\in W'$\ and any $y_1,y_2\in W'$\ where $y_1$\ and $y_2$\ have distance less than $d_0$\ (here the distance is that of the ambient space $X$), 
	\begin{eqnarray*}
	\lefteqn{|h_{\Omega,\,W}(t,x,y_1)-h_{\Omega,\,W}(t,x,y_2)|}\\
	&\leq& |h_{\Omega}(t,x,y_1)-h_{\Omega}(t,x,y_2)|+2\sup_{y\in W'}\int_{t-\delta}^{t}\int_{\partial U}h_\Omega(t-t',z,y)\,\mu_x(dt'dz)\\
	&&\hspace{-.2in}+\left|\int_{0}^{t-\delta}\int_{\partial W}h_\Omega(t-t',z,y_1)\,\mu_x(dt'dz)-\int_{0}^{t-\delta}\int_{\partial W}h_\Omega(t-t',z,y_2)\,\mu_x(dt'dz)\right|\\
	&<&4\epsilon.
	\end{eqnarray*}
	Note that $d_0$\ is independent of $x\in W'$; in other words, for any $t\in (0,c)$, $h_{\Omega,\,W}(t,x,y)$\ is equicontinuous in $y$\ on $W'$. By the symmetry of $h_{\Omega,\,W}$, it is also equicontinuous in $x$\ on $W'$. We conclude that for any $t\in (0,c)$, $h_{\Omega,\,W}(t,x,y)$\ is continuous on $W'\times W'$\ and thus on $W\times W$. Moreover, $h_{\Omega,\,W}$\ is smooth in $t$ (it is in $C^\infty((0,c)\rightarrow L^\infty(W\times W))$\ and hence in $C^\infty((0,c)\rightarrow C(W\times W))$), so $h_{\Omega,\,W}$\ is continuous on $(0,c)\times W\times W$. This completes the proof of the claim that $u$\ is continuous on $I\times V$.
\section{Applications}
\setcounter{equation}{0}
\subsection{Existence and local boundedness of density under local ultracontractivity assumption}
In this subsection we discuss the existence of the heat kernel of a locally ultracontractive semigroup. We start with a more general theorem.
\begin{theorem}
\label{densitythm}
Let $(X,m)$\ be a metric measure space and $\Omega\subset X$\ be an open subset. Let $(\mathcal{E}_\Omega,\mathcal{F}_\Omega)$\ be a symmetric local regular Dirichlet form on $L^2(\Omega,\mu_\Omega)$, where $\mu_\Omega$\ is some measure on $\Omega$\ that admits the same measure-zero sets but not necessarily agrees with the restriction of the measure $m$.  Suppose $(\mathcal{E}_\Omega,\mathcal{F}_\Omega)$\ satisfies
\begin{itemize}
    \item Assumption \ref{cutoff} (existence of nice cut-off functions);
    \item the corresponding semigroup $(H_t^\Omega)_{t>0}$\ is locally ultracontractive (\ref{localultraassumption});
    \item $(H_t^\Omega)_{t>0}$\ satisfies the $L^\infty$\ off-diagonal upper bound (\ref{globalGaussian})(\ref{mcontrolpoly}) in $\Omega$.
\end{itemize}
Let $I:=(0,1)$. Let $(S_t)_{t>0}$\ be a family of continuous linear operators from $L^2(X,m)$\ to $L^2_{\scaleto{\mbox{loc}}{5pt}}(\Omega,\mu_\Omega)$\ that satisfies
\begin{itemize}
    \item the map $L^2(X,m)\rightarrow L^2_{\scaleto{\mbox{loc}}{5pt}}(I\times\Omega,\,dtd\mu_\Omega)$, $g\mapsto S_tg$\ is continuous;
    \item for any $g\in L^2(X,m)$, the restriction of $u(t,x):=S_tg(x)$\ on $I\times \Omega$\ is a local weak solution of $(\partial_t+P_\Omega)u=0$\ on $I\times \Omega$. $-P_\Omega$\ denotes the generator of $(\mathcal{E}_\Omega,\mathcal{F}_\Omega)$.
\end{itemize}
Then there exists a measurable function $s(t,x,y)$\ defined a.e. on $I\times \Omega\times X$, satisfying
%\begin{eqnarray}
%\esssup_{t\in(0,1)\,x\in \Omega}\int_Xs^2(t,x,y)\,dm(y)<+\infty,
%\end{eqnarray}
that the function
$(t,x)\mapsto ||s^{t,x}||_{L^2(X)}$\ belongs to $L_{\scaleto{\mbox{loc}}{5pt}}^\infty(I\times\Omega)$,
and that for any $g\in L^2(X,m)$,
\begin{eqnarray*}
		S_tg(x)=\int_X s(t,x,y)g(y)\,dm(y).
	\end{eqnarray*}
\end{theorem}
\begin{proof}
%For any $U\Subset \Omega$, consider the restriction semigroup $H_t^{\Omega,\,U}$\ and its corresponding generator $-P_{\Omega,\,U}$\ (see Example \ref{restrictionformeg}). Then $H_t^{\Omega,\,U}$\ is globally ultracontractive on $U$\ and the sets of local weak solutions of $(\partial_t+P_\Omega)u=0$\ and $(\partial_t+P_{\Omega,\,U})u=0$\ on $I\times U$\ agree. 
By the proofs of Theorems \ref{mainthm} and \ref{mainthm0} (see (\ref{keyest})), we know that for any $J\Subset I$, $V\Subset \Omega$, there exist a constant $C(J,V)>0$\ and nice product cut-off functions $\overline{\eta},\overline{\Psi}$\ that equal to $1$\ on $J\times V$\ and are compactly supported in $I\times \Omega$, such that for any local weak solution $u$\ of $(\partial_t+P_\Omega)u=0$\ on $I\times \Omega$,
\begin{eqnarray}
\label{5.1.1}
||u||_{L^\infty(J\times V)}\leq C(J,V)\left[||\overline{\Psi}u||_{L^2(I\rightarrow \mathcal{F}_\Omega)}+||\overline{\eta}u||_{L^2(I\rightarrow \mathcal{F}_\Omega)}\right].
\end{eqnarray}
Note that the norm $||\cdot||_{L^2(I\rightarrow \mathcal{F}_\Omega)}$\ is given by
\begin{eqnarray*}
||v||_{L^2(I\rightarrow \mathcal{F}_\Omega)}=\left(\int_I\mathcal{E}_\Omega(v^t,v^t)\,dt+\int_I\int_\Omega v^2\,d\mu_\Omega dt\right)^{1/2},
\end{eqnarray*}
for any $v\in L^2(I\rightarrow \mathcal{F}_\Omega)$. In the rest of the proof we write $||\cdot||_{L^2(I\rightarrow \mathcal{F})}$\ instead for short.
By assumption, $u:=S_tg$\ is a local weak solution of $(\partial_t+P_\Omega)u=0$\ on $I\times \Omega$. We show that for this particular $u$, $||\overline{\Psi}u||_{L^2(I\rightarrow \mathcal{F})}\lesssim ||g||_{L^2(X,m)}$. The estimate for $||\overline{\eta}u||_{L^2(I\rightarrow \mathcal{F})}$\ is similar. Recall that we use $I_{\overline{\Psi}}\times U_{\overline{\Psi}}$\ to represent an open set that contains the support of $\overline{\Psi}$\ and is precompact in $I\times\Omega$. Since $\overline{\Psi}$\ is a nice product cut-off function,
\begin{eqnarray*}
\lefteqn{||\overline{\Psi}u||^2_{L^2(I\rightarrow \mathcal{F})}}\\
&\leq& K\left(\int_I\int_\Omega\overline{\Psi}^2\,d\Gamma(u,u)\,dt+\int_I\int_\Omega(\overline{\Psi}u)^2\,dkdt+\int_{I_{\overline{\Psi}}}\int_{U_{\overline{\Psi}}}u^2\,d\mu_\Omega dt\right)
\end{eqnarray*}
for some constant $K=K(\overline{\Psi})>0$. Here $\Gamma,k$\ denote the energy measure and killing measure associated with $(\mathcal{E}_\Omega,\mathcal{F}_\Omega)$. Moreover, by the proof of Proposition 6.9 in \cite{L2}, there exist some open set $I'\times U'$\ with $I_{\overline{\Psi}}\times U_{\overline{\Psi}}\Subset I'\times U'\Subset I\times\Omega$\ and some constant $K'(I_{\overline{\Psi}},U_{\overline{\Psi}},I',U')>0$, such that
\begin{eqnarray*}
\int_I\int_\Omega\overline{\Psi}^2\,d\Gamma(u,u)\,dt+\int_I\int_\Omega(\overline{\Psi}u)^2\,dkdt\leq K'\int_{I'}\int_{U'}u^2\,d\mu_\Omega dt.
\end{eqnarray*}
Altogether, for some $K_0>0$,
\begin{eqnarray}
\label{5.1.2}
||\overline{\Psi}u||^2_{L^2(I\rightarrow \mathcal{F})}\leq K_0\int_{I'}\int_{U'}u^2\,d\mu_\Omega dt.
\end{eqnarray}
Moreover, the map $g\mapsto S_tg$\ by assumption is continuous from $L^2(X,m)$\ to $L^2_{\scaleto{\mbox{loc}}{5pt}}(I\times\Omega,\,dtd\mu_\Omega)$, hence for $I'\times U'\Subset I\times \Omega$, there exists some $C'(I',U')>0$\ such that
\begin{eqnarray}
\left(\int_{I'}\int_{U'}(S_tg)^2\,d\mu_\Omega dt\right)^{1/2}\leq C'(I',U')||g||_{L^2(X,m)}. \label{5.1.3}
\end{eqnarray}
From (\ref{5.1.1}), (\ref{5.1.2}), and (\ref{5.1.3}), we conclude that
\begin{eqnarray*}
||S_tg||_{L^\infty(J\times V)}=||u||_{L^\infty(J\times V)}\leq M(J,V)||g||_{L^2(X,m)}
\end{eqnarray*}
for some constant $M(J,V)>0$. By the Dunford-Pettis Theorem (cf. e.g. \cite{DunfordPettis}), there exists a function $s_{\scaleto{J\times V}{5pt}}(t,x,y)$\ defined a.e. on $J\times V\times X$, satisfying that
\begin{eqnarray*}
\esssup_{(t,x)\in J\times V}\left(\int s_{\scaleto{J\times V}{5pt}}(t,x,y)^2\,dm(y)\right)^{1/2}\leq M(J,V),
\end{eqnarray*}
such that for any $g\in L^2(X,m)$, for a.e. $(t,x)\in J\times V$,
\begin{eqnarray*}
S_tg(x)=\int_X s_{\scaleto{J\times V}{5pt}}(t,x,y)g(y)\,dm(y).
\end{eqnarray*}
Since the open subset $J\times V\Subset I\times\Omega$\ is arbitrary, we conclude that there exists a function $s(t,x,y)$\ defined a.e. on $I\times\Omega\times X$, satisfying the two conditions stated in the theorem. That is, the function
$(t,x)\mapsto ||s^{t,x}||_{L^2(X,m)}$\ belongs to $L_{\scaleto{\mbox{loc}}{5pt}}^\infty(I\times\Omega)$; for any $g\in L^2(X,m)$, $S_tg(x)=\int_X s(t,x,y)g(y)\,dm(y)$.
\end{proof}
\begin{remark}
When $\Omega=X$\ and $S_t$\ is self-adjoint, the density function $s(t,x,y)$\ is locally bounded on $I\times X\times X$. In particular, Theorem \ref{densitythm} applied to $S_t=H_t$\ and $\mu=m$\ leads to the conclusion that local ultracontractivity of the semigroup $H_t$, together with the off-diagonal upper bound of $H_t$\ and the existence of nice cut-off functions, imply the existence of a locally bounded heat kernel.
\end{remark}
\subsection{Ancient local weak solutions}
One other application of the main theorem and especially the estimate (\ref{keyest}) is the following proposition which improves the preceding $L^2$\ structure result in \cite{L2} of ancient local weak solutions of heat equations to a pointwise result, under stronger conditions. See \cite{app1,app2analyticity} for structure results for ancient solutions of classical heat equations on Riemannian manifolds. Let $(X,m,\mathcal{E},\mathcal{F})$\ be a Dirichlet space where the Dirichlet form is symmetric regular local. Let $-P$\ be the generator of the Dirichlet form. An {\em ancient local weak solution} $u$\ of $(\partial_t+P)u=0$\ is any local weak solution of the heat equation on $(-\infty,b)\times X$\ for some $b>0$. We first review an additional assumption we made on the existence of cut-off functions in \cite{L2}. Intuitively, this assumption says that there is a sequence of increasing open sets far enough apart, so that there are nice cut-off functions for each pair of adjacent open sets with small enough $C_2$'s. The precise assumption (taking $C_2=1$) is as follows. Note that when the space $X$\ is compact, this assumption is automatically satisfied by taking all sets to be $X$.
\begin{assumption}
	\label{appassumption}
	For any $C_1>0$, there exist
	\begin{itemize}
		\item[(i)] an exhaustion of $X$, $\{W_{C_1,i}\}_{i\in \mathbb{N}_+}=:\{W_{i}\}_{i\in \mathbb{N}_+}$. That is, $\{W_{i}\}_{i\in \mathbb{N}_+}$\ is a sequence of increasing open sets, satisfying
		\begin{eqnarray*}
		W_{i}\Subset W_{i+1},\ \bigcup_{i=1}^{\infty}W_{i}=X.
		\end{eqnarray*}
		\item[(ii)] a sequence of cut-off functions $\{\varphi_{i}\}_{i\in \mathbb{N}_+}$, where each $\varphi_i$\ is a cut-off function for the pair $W_{i}\subset W_{i+1}$, i.e., $\varphi_i=1$\ on $W_{i}$, $\mbox{supp}\{\varphi_i\}\subset W_{i+1}$. $\varphi_i$\ further satisfies that for any $v\in \mathcal{F}$,
		\begin{eqnarray}
		\label{cut-offappineq}
		\int_Xv^2\,d\Gamma(\varphi_i,\varphi_i)\leq C_1 \int_X \varphi_i^2\,d\Gamma(v,v)+\int_{\scaleto{\mbox{supp}\{\varphi_i\}}{5pt}}v^2\,dm.
		\end{eqnarray}
	\end{itemize}
\end{assumption}
Under this further assumption, when the conditions in Theorem \ref{modelthm} are satisfied, we show the following pointwise structure theorem on ancient local weak solutions with exponential growth of the heat equation.
\begin{proposition}
Let $(X,m)$\ be a metric measure space and $(\mathcal{E},\mathcal{F})$\ be a symmetric regular local Dirichlet form on $X$. Assume that the Dirichlet space $(X,\mathcal{E},\mathcal{F})$\ satisfies Assumption \ref{cutoff}, and when $X$\ is not compact, further satisfies Assumption \ref{appassumption}. Let $(H_t)_{t>0}$\ and $-P$\ be the corresponding semigroup and generator. Assume the semigroup is locally ultracontractivite (\ref{localultraassumption}) and satisfies the $L^\infty$\ off-diagonal upper bound (\ref{globalGaussian})(\ref{mcontrolpoly}).

Let $u$\ be any ancient local weak solution of $(\partial_t+P)u=0$. Suppose $u$\ satisfies the $L^2$\ exponential growth condition, namely, there exists some $c_u>0$, such that for any $T>1$, any $i\in \mathbb{N}_+$,
\begin{eqnarray}
\label{expgrowthbound}
\int_{[-T,0]\times W_i}|u(t,x)|^2\,dmdt\leq e^{c_u(T+i)}.
\end{eqnarray}
Then $u$\ is analytic in $t\in (-\infty,0]$, in the sense that
\begin{eqnarray*}
\sum_{i=0}^{k}\frac{\partial_t^iu(0,x)}{i!}t^i\rightarrow u(t,x)
\end{eqnarray*}
in $L^\infty_{\scaleto{\mbox{loc}}{5pt}}((-\infty,0]\times X)$\ as $k\rightarrow \infty$.
\end{proposition}
\begin{proof}
By Theorem \ref{modelthm} (a special case of Theorem \ref{mainthm}), $u$\ locally belongs to $C^\infty((-\infty,0]\rightarrow L^\infty(X))$. By the Taylor expansion formula in $t$, for any $t<0$,
\begin{eqnarray*}
	u(t,x)=\sum_{i=0}^{k}\frac{\partial_t^iu(0,x)}{i!}t^i+\int_{0}^{t}\partial_s^{k+1}u(s,x)\frac{(t-s)^k}{k!}\,ds
\end{eqnarray*}
in $L_{\scaleto{\mbox{loc}}{5pt}}^\infty(X)$. This implies that for any precompact open set $V\subset X$\ and any $I=(-T,0]$\ where $T>0$,
\begin{eqnarray}
\label{5.2.1}
\esssup_{(t,x)\in I\times V}\left|u(t,x)-\sum_{i=0}^{k}\frac{\partial_t^iu(0,x)}{i!}t^i\right|\leq\frac{T^{k+1}}{k!}\esssup_{(t,x)\in I\times V}\left|\partial_s^{k+1}u(s,x)\right|.
\end{eqnarray}
As in the proof of Theorem \ref{densitythm}, estimate (\ref{keyest}) together with Proposition 6.8 in \cite{L2} indicate that for $I\times V$, there exist some open set $U$\ with $V\Subset U\Subset X$\ and some finite interval $I'=(-T',0]$\ with $0<T<T'$, there exists some constant $C(T,T',V,U)>0$, such that
\begin{eqnarray*}
\lefteqn{\esssup_{(s,x)\in I\times V}\left|\partial_s^{k+1}u(s,x)\right|
\leq C(T,T',V,U)\left(\int_{I'\times U}(\partial_s^{k+1}u(s,x))^2\,dmds\right)^{1/2}}\\
&\leq& C(T,T',V,U)\cdot 20^{k+1}\left(\int_{(T'-2k-2,\,0]\times W_{n_0+3k+3}}(u(s,x))^2\,dmds\right)^{1/2}.
\end{eqnarray*}
Here $n_0$\ is some integer so that $U\subset W_{n_0}$. Applying the exponential growth bound (\ref{expgrowthbound}), we have for some $C>0$,
\begin{eqnarray*}
\esssup_{(s,x)\in I\times V}\left|\partial_s^{k+1}u(s,x)\right|<C^{k+1}e^{c_u(|T'|+n_0+5k+5)}.
\end{eqnarray*}
It follows that the upper bound in (\ref{5.2.1}) tends to $0$\ as $k$\ tends to infinity. Hence the proposition holds.
\end{proof}
\section{Appendix}
\setcounter{equation}{0}
In this section we discuss one case where the $L^\infty$\ off-diagonal upper bound follows from the ultracontractivity property of the semigroup. Let $(X,m)$\ be a metric measure space as before. Let $(\mathcal{E},\mathcal{F})$\ be a symmetric regular local Dirichlet form on $L^2(X,m)$. Assume that the Dirichlet space satisfies Assumption \ref{cutoff} (existence of nice cut-off functions), and that there exists a distance function $\rho_X$\ on $X$\ satisfying
\begin{itemize}
\item $\rho_X$\ is continuous and defines the topology of $X$;
\item there exist two numbers $\alpha\geq 0,\,\beta>0$\ such that for any $V\Subset U\Subset X$, any $0<C_1<1$, there exists a nice cut-off function $\eta$\ for the pair $V\subset U$, corresponding to constants $C_1$\ and $C_2=C_1^{-\alpha}\rho_X(V,U^c)^{-\beta}$. That is, $\eta\equiv 1$\ on $V$, $\mbox{supp}\{\eta\}\subset U$, and for any $v\in \mathcal{F}$,
\begin{eqnarray*}
\int_Xv^2\,d\Gamma(\eta,\eta)\leq C_1\int_X\eta^2\,d\Gamma(v,v)+C_1^{-\alpha}\rho_X(V,U^c)^{-\beta}\int_{\scaleto{\mbox{supp}\{\eta\}}{5pt}}v^2\,dm.
\end{eqnarray*}
Here for any two measurable sets $\Omega_1,\Omega_2$, $\rho_X(\Omega_1,\Omega_2)$\ denotes the distance between the two sets induced by the pointwise distance $\rho_X$.
\end{itemize}
\begin{theorem}
\label{appendixprop}
Under the above hypotheses, assume further that the Dirichlet form's corresponding semigroup $H_t$\ is globally ultracontractive. More precisely, assume that there exist some interval $(0,T)$\ with $T>0$\ and some continuous non-increasing function $M(t):(0,T)\rightarrow \mathbb{R}_+$\ satisfying
\begin{eqnarray}
\label{M(t)}
\lim_{t\rightarrow 0^+}t^{\frac{1}{1+2\alpha}}M(t)=0,
\end{eqnarray}
such that
\begin{eqnarray*}
||H_t||_{L^2(X)\rightarrow L^\infty(X)}\leq e^{M(t)}.
\end{eqnarray*}
Then the following $L^\infty$\ off-diagonal upper bound holds: for any two open sets $U,V\Subset X$\ with $\rho_X(U,V)>0$, for any $u,v\in L^1(X)$\ with $\mbox{supp}\{u\}\subset U$, $\mbox{supp}\{v\}\subset V$,
\begin{eqnarray}
\left|<H_{2t}u,\,v>\right|\leq C_0\exp{\left\{-\frac{1}{4}\left(\frac{b^{\beta}\rho_X(U,V)^\beta}{4^{1+\alpha+\beta}ct}\right)^{\frac{1}{1+2\alpha}}\right\}}||u||_{L^1(U)}||v||_{L^1(V)}.
\end{eqnarray}
Here $C_0>0$\ is some constant independent of $u,v$; $b,c>0$\ are two small constants satisfying 
\begin{eqnarray*}
c\leq \left((1+4^{1+\alpha})\sum\limits_{m=1}^{\infty}m^{-(2\beta+1+2\alpha)}\right)^{-1};
\end{eqnarray*}
\begin{eqnarray*}
\left(1-b\sum_{m=1}^{\infty}m^{-2}\right)^\beta\left(1-c\sum_{m=1}^{\infty}m^{-(2\beta+1+2\alpha)}\right)^{-1/(1+2\alpha)}\geq 2.
\end{eqnarray*}
\end{theorem}

In general, when the semigroup is only locally ultracontractive, we may consider the semigroup corresponding to a restricted form on some precompact open subset.
\subsection{Proof of Theorem \ref{appendixprop}}
From Lemma 8.1 in \cite{L2}, by taking $C_2=C_1^{-\alpha}\rho_X^{-\beta}$\ there (see Remark 8.2), approximation by precompact open sets, and the continuity of the distance function $\rho_X$, we know that for any two measurable sets $U,V\Subset X$\ with $\rho_X(U,V)>0$, for any $u,v\in L^2(X)$\ with $\mbox{supp}\{u\}\subset U$, $\mbox{supp}\{v\}\subset V$,
\begin{eqnarray}
\label{lem}
|<H_su,\,v>|\leq \exp{\left\{-\left(\frac{\rho_X(U,V)^\beta}{4^{1+\alpha}s}\right)^{1/\left(1+2\alpha\right)}\right\}}\left|\left|u\right|\right|_2\left|\left|v\right|\right|_2.
\end{eqnarray}
In below we use (\ref{lem}) and iterations to obtain two lemmas that comprise the key part of the proof for Theorem \ref{appendixprop}. The first lemma is as follows.
\begin{lemma}
	\label{lem1}
	Let $U,V$\ be precompact open sets with distance $\rho_X(U,V)=:2d$. Let $f,g$\ be functions satisfying that $\mbox{supp}\{f\}\subset U$, $\mbox{supp}\{g\}\subset V$, $f\in L^1(U)$, $g\in L^2(V)$. Then there exists some $\widetilde{V}$\ with $V\subset \widetilde{V}$\ and $\rho_X(U,\widetilde{V})\geq \left(1+\nu\right)d$, there exists some $\widetilde{g}$\ with $\mbox{supp}\left\{\widetilde{g}\right\}\subset \widetilde{V}$\ and $\left|\left|\widetilde{g}\right|\right|_{L^2(\widetilde{V})}\leq \left|\left|g\right|\right|_{L^2(V)}$, such that for all $0<t<t_0$\ where $0<t_0<1$\ is some fixed number,
	\begin{eqnarray}
	\left|\left<H_tf,\,g\right>\right|\leq C\exp{\left\{-\left(\frac{b^\beta d^\beta}{4^{1+\alpha}ct}\right)^{\frac{1}{1+2\alpha}}\right\}}||f||_1||g||_2+\left|\left<H_{\delta t}f,\,\widetilde{g}\right>\right|.\label{lem2}
	\end{eqnarray}
    Here $b,c>0$\ are any small enough numbers satisfying
	\begin{eqnarray*}
	\sum_{m=1}^{+\infty}\frac{b}{m^2}<1;\ \ c\leq \left(5\sum\limits_{m=1}^{\infty}\frac{1}{m^{2\beta+1+2\alpha}}\right)^{-1}.
    \end{eqnarray*}
	$C$\ is a constant that only depends on $d,b,c$. $\delta$, $\nu$\ are defined as
	\begin{eqnarray*}
	\delta:=1-\sum_{m=1}^{+\infty}\frac{c}{m^{2\beta+1+2\alpha}},\ \  \nu:=1-\sum_{m=1}^{+\infty}\frac{b}{m^2}.
	\end{eqnarray*}
\end{lemma}
\begin{proof}
	We use iteration to decompose $<H_tf,\,g>$\ into a sum of terms in the form of (\ref{lem}), and a remaining term.
	\subsubsection*{Step 1 of iteration}
	For some small number $b>0$\ to be determined later, there exists a measurable set $V_1$\ satisfying that
	\begin{eqnarray*}
	V\subset V_1,\ \ \rho_X\left(X\setminus V_1,V\right)\geq bd,\ \ \rho_X\left(U,V_1\right)\geq \left(2-b\right)d.
	\end{eqnarray*} 
    For example, $V_1$\ can be taken as
    \begin{eqnarray*}
    V_1=\left\{x\in X\,|\,\rho_X(x,V)\leq bd\right\}.
    \end{eqnarray*}
    Denote $U_1:=X\setminus V_1$, then $U\subset U_1$\ and $\rho_X\left(U_1,V\right)\geq bd$. Let $\Phi_1$\ and $\Psi_1$\ be the characteristic functions of $U_1$\ and $V_1$\ respectively. For any $0<c<1$,
	\begin{eqnarray}
	\lefteqn{<H_tf,\,g>= <H_{ct}\left(\Phi_1+\Psi_1\right)H_{(1-c)t}f,\,g>}\notag\\
	&=& <H_{ct}\left(\Phi_1 H_{(1-c)t}f\right),\,g>+<H_{(1-c)t}f,\,\Psi_1H_{ct}g>. \label{step1}
	\end{eqnarray}
	This is the first step in splitting $<H_tf,\,g>$. To proceed we give an estimate for the first term in (\ref{step1}), and further split the second term in the second iteration step. For the first term in (\ref{step1}), apply (\ref{lem}) for $u=\Phi_1 H_{(1-c)t}f$, $v=g$, and $s=ct$, we get that
	\begin{eqnarray*}
	\lefteqn{|<H_{ct}\left(\Phi_1 H_{(1-c)t}f\right),\,g>|}\notag\\
	&\leq& \exp{\left\{-\left(\frac{\rho_X\left(U_1,V\right)^\beta}{4^{1+\alpha}ct}\right)^{1/\left(1+2\alpha\right)}\right\}}\left|\left|\Phi_1 H_{(1-c)t}f\right|\right|_2\left|\left|g\right|\right|_2\notag\\
	&\leq& \exp{\left\{-\left(\frac{b^\beta d^\beta}{4^{1+\alpha}ct}\right)^{1/\left(1+2\alpha\right)}\right\}}\cdot e^{M((1-c)t)}||f||_1||g||_2\notag\\
	&\leq& \exp{\left\{-\left(\frac{b^\beta d^\beta}{4^{1+\alpha}ct}\right)^{1/\left(1+2\alpha\right)}\right\}}\exp{\left\{\frac{\epsilon}{\left((1-c)t\right)^{1/\left(1+2\alpha\right)}}\right\}}||f||_1||g||_2. 
	\end{eqnarray*}
	The last line holds for all $0<t<t_0$\ for some fixed $t_0>0$\ that depends on the choice of $\epsilon$. More precisely, recall the assumption that $M(t)=o\left(t^{-1/\left(1+2\alpha\right)}\right)$. Thus for any $\epsilon>0$,
	\begin{eqnarray*}
	\lim_{t\rightarrow 0}\frac{M(t)}{\epsilon t^{-1/\left(1+2\alpha\right)}}=0.
	\end{eqnarray*}
    $\epsilon$\ is to be determined later.
    \subsubsection*{Step 2 of iteration}
	Let $\iota:=2\beta+1+2\alpha$. Let
	\begin{eqnarray*}
	g_1:=\Psi_1H_{ct}g,
	\end{eqnarray*}
	then $g_1$\ is supported in $V_1$, and $\left|\left|g_1\right|\right|_2\leq \left|\left|g\right|\right|_2$\ since $H_{ct}$\ is a contraction on $L^2(X)$. We repeat the iteration to the second term in (\ref{step1}) by writing
	\begin{eqnarray}
	\lefteqn{<H_{(1-c)t}f,\,\Psi_1H_{ct}g>
	= <H_{(1-c)t}f,\,g_1>}\notag\\
	&=& <H_{ct/2^{\iota}}\left(\Phi_2+\Psi_2\right)H_{t(1-c-c/2^{\iota})}f,\, g_1>\notag\\
	&=& <H_{ct/2^{\iota}}\left(\Phi_2H_{t(1-c-c/2^{\iota})}f\right),\, g_1>\notag\\
	&&+<H_{t(1-c-c/2^{\iota})}f,\, \Psi_2H_{ct/2^{\iota}}g_1>.\label{step2}
	\end{eqnarray}
	Here $U_2$, $V_2$\ ($U_2=X\setminus V_2$) are such that $U_2\subset U_1$, $V_1\subset V_2$, and
	\begin{eqnarray*}
	\rho_X(U_2,V_1)\geq bd/2^2,\ \ \rho_X(U,V_2)\geq (2-b-b/2^2)d.
	\end{eqnarray*}
	$\Phi_2$, $\Psi_2$\ are characteristic functions of $U_2$, $V_2$, respectively. As in the first iteration step, we can estimate the first term in (\ref{step2}) using (\ref{lem}) as
	\begin{eqnarray*}
	\lefteqn{\left|<H_{ct/2^{\iota}}\left(\Phi_2H_{t(1-c-c/2^{\iota})}f\right),\, g_1>\right|}\notag\\
	&\leq& \exp{\left\{-\left(\frac{b^\beta d^\beta/2^{2\beta}}{4^{1+\alpha}ct/2^{\iota}}\right)^{\frac{1}{1+2\alpha}}\right\}}\exp{\left\{M\left(t\left(1-c-\frac{c}{2^{\iota}}\right)\right)\right\}}\left|\left|f\right|\right|_1\left|\left|g_1\right|\right|_2\notag\\
	&\leq& \exp{\left\{-2\left(\frac{b^\beta d^\beta}{4^{1+\alpha}ct}\right)^{\frac{1}{1+2\alpha}}\right\}}\exp{\left\{\epsilon t^{-\frac{1}{1+2\alpha}}\left(1-c-\frac{c}{2^{\iota}}\right)^{-\frac{1}{1+2\alpha}}\right\}}||f||_1||g||_2.
	\end{eqnarray*}
	\subsubsection*{General step of iteration}
	 Let $s_k:=\sum\limits_{m=1}^{k}m^{-\iota}$. Repeating the above iteration, in the general nth step we have
	\begin{eqnarray}
	\lefteqn{<H_tf,\,g>}\notag\\
	&=&\sum_{k=1}^{n}\left<H_{ct/k^{\iota}}\left(\Phi_kH_{t(1-cs_k)}f\right),\,g_{k-1}\right>+\left<H_{t(1-cs_n)}f,\,\Psi_nH_{ct/n^{\iota}}g_{n-1}\right>\notag\\
	&=&\sum_{k=1}^{n}\left<H_{ct/k^{\iota}}\left(\Phi_kH_{t(1-cs_k)}f\right),\,g_{k-1}\right>+\left<H_{t(1-cs_n)}f,\,g_n\right>,\label{general}
	\end{eqnarray}
	where each $\Phi_k,\Psi_k$\ is a pair of characteristic functions corresponding to some $U_k,V_k$\ that partition $X$. The sets satisfy that $U_k\subset U_{k-1}$, $V_{k-1}\subset V_k$, and
	\begin{eqnarray*}
	\rho_X\left(U_k,V_{k-1}\right)\geq \frac{bd}{k^2},\ \ \rho_X\left(U,V_k\right)\geq \left(2-\sum_{m=1}^{k}\frac{b}{m^2}\right)d.
	\end{eqnarray*} 
	The functions $g_k$\ are obtained from $g_{k-1}$\ by
	\begin{eqnarray*}
	g_k=\Psi_kH_{t(1-cs_k)}g_{k-1},
	\end{eqnarray*}
	and here $g_0:=g$. In particular, all $\left|\left|g_k\right|\right|_2\leq \left|\left|g\right|\right|_2$.
	
	Next we find an upper bound for the sum term in (\ref{general}), and then treat the remaining term in (\ref{general}).
	
	\subsubsection*{Estimate of the sum term}
	As the estimates shown in the first two iteration steps, the sum in (\ref{general}) is bounded by
	\begin{eqnarray*}
	\lefteqn{\left|\sum_{k=1}^{n}\left<H_{ct/k^{\iota}}\left(\Phi_kH_{t(1-cs_k)}f\right),\,g_{k-1}\right>\right|}\\
	&\leq& \sum_{k=1}^{n}||f||_1||g||_2\cdot\exp{\left\{-k\left(\frac{b^\beta d^\beta}{4^{1+\alpha}ct}\right)^{\frac{1}{1+2\alpha}}+\epsilon t^{-\frac{1}{1+2\alpha}}\left(1-cs_k\right)^{-\frac{1}{1+2\alpha}}\right\}}.
	\end{eqnarray*}
	We want to pick $\epsilon,c$\ so that for all $k\in \mathbb{N}_+$,
	\begin{eqnarray}
	k\left(\frac{b^\beta d^\beta}{4^{1+\alpha}ct}\right)^{\frac{1}{1+2\alpha}}\geq 2\epsilon\, t^{-\frac{1}{1+2\alpha}}\left(1-cs_k\right)^{-\frac{1}{1+2\alpha}}. \label{c}
	\end{eqnarray}
	%which is equivalent to
	%\begin{eqnarray*}
	%c\leq \min_{k}\frac{k^{1+2\alpha}b^\beta d^\beta}{k^{1+2\alpha}b^\beta d^\beta s_k+4(2\epsilon)^{1+2\alpha}}.
	%\end{eqnarray*}
	Choose $\epsilon$\ so that $\left(2\epsilon\right)^{1+2\alpha} = b^\beta d^\beta$,
    then (\ref{c}) is equivalent to
    \begin{eqnarray*}
    c\leq \min_{k}\frac{k^{1+2\alpha}}{k^{1+2\alpha}s_k+4^{1+\alpha}},
    \end{eqnarray*}
    thus it suffices to pick $0<c< \left((1+4^{1+\alpha})s_\infty\right)^{-1}$, where
    $s_\infty=\lim\limits_{k\rightarrow\infty}s_k=\sum\limits_{m=1}^{\infty}\frac{1}{m^{2\beta+1+2\alpha}}$. It follows that for any such small enough $c>0$, any $1\leq k\leq n$,
	\begin{eqnarray}
	\hspace{-.15in}\left|\left<H_{ct/k^{\iota}}\left(\Phi_kH_{t(1-cs_k)}f\right),\,g_{k-1}\right>\right|
	%&\leq& \sum_{k=1}^{n}\exp{\left\{-k\left(\frac{b^\beta d^\beta}{4ct}\right)^{\frac{1}{1+2\alpha}}\right\}}\exp{\left\{\epsilon t^{-\frac{1}{1+2\alpha}}\left(1-\sum_{m=1}^{k}c/m^{2\beta+1+2\alpha}\right)^{-\frac{1}{1+2\alpha}}\right\}}||f||_1||g||_2\notag\\
	\leq \exp{\left\{-\frac{k}{2}\left(\frac{b^\beta d^\beta}{4^{1+\alpha}ct}\right)^{\frac{1}{1+2\alpha}}\right\}}||f||_1||g||_2.\label{firsttermestimate}
	\end{eqnarray}
	\subsubsection*{Treatment of the remaining term}
	For the remaining term in (\ref{general}), $\left<H_{t(1-cs_n)}f,\,g_n\right>$, note that
	\begin{itemize}
	    \item $||g_n||_2\leq ||g||_2$\ for all $n$\ implies the existence of a weakly convergent subsequence of $\{g_n\}$\ in $L^2(X)$. Denote the weak limit by $\widetilde{g}$. Then $||\widetilde{g}||_2\leq ||g||_2$.
	    \item $H_{t\,\left(1-cs_n\right)}f$\ converges to $H_{t\,\left(1-cs_\infty\right)}f$\ in $L^2(X)$, as $n\rightarrow \infty$.
	\end{itemize}
	Hence for the subsequence mentioned above,
	\begin{eqnarray}
	<H_{t(1-cs_n)}f,\,g_n>\rightarrow <H_{t(1-cs_\infty)}f,\,\widetilde{g}>.\label{remainingterm}
	\end{eqnarray}
	Letting $\delta:=1-cs_\infty$, we obtain the second term in (\ref{lem2}).
	
	Recall that each $g_n$\ is supported in a set $V_n$, with $V_1\subset V_2\subset \cdots \subset V_n\subset \cdots$. Consider
	\begin{eqnarray*}
	\widetilde{V}:=\bigcup_n V_n,
	\end{eqnarray*}
	then the weak limit $\widetilde{g}$\ is supported in $\widetilde{V}$. Let $\widetilde{U}$\ be the intersection of the sequence $U_1\supset U_2\supset \cdots \supset U_n\supset \cdots $, i.e.,
	\begin{eqnarray*}
	\widetilde{U}:=\bigcap_n U_n,
	\end{eqnarray*}
	then $U\subset \widetilde{U}$, $V\subset \widetilde{V}$, and $\rho_X(U,\widetilde{V})\geq \left(2-\sum\limits_{m=1}^{+\infty}b/m^2\right)d=\left(1+\nu\right)d$. Here
	\begin{eqnarray*}
	\nu:=1-\sum_{m=1}^{+\infty}b/m^2;
	\end{eqnarray*}
    $b>0$\ is chosen so that $\sum\limits_{m=1}^{+\infty}b/m^2<1$.
	
	%Since each $g_n$\ is supported in $V_n$, all $g_n$\ are supported in $\widetilde{V}$. From the definition of $g_n$\ we know that $\left|\left|g_n\right|\right|_2\leq \left|\left|g\right|\right|_2$\ for all $n$. Hence by Banach-Alaoglu Theorem (and note that $L^2(\widetilde{V})$\ is reflexive), there exists a subsequence of $\left\{g_n\right\}$\ that converges weakly. Let $\widetilde{g}\in L^2(\widetilde{V})$\ be the weak limit. In particular,
%	\begin{eqnarray*}
%	\left|\left|\widetilde{g}\right|\right|_2\leq \left|\left|g\right|\right|_2.
%	\end{eqnarray*}
%	Let $\Omega$\ in the discussion above be $\widetilde{V}$, and $\widetilde{g}$\ above as this weak limit $\widetilde{g}$, we get (\ref{remainingterm}).
	
	We are now ready to prove (\ref{lem2}). For any $0<t<t_0$, applying (\ref{general}) and  (\ref{firsttermestimate}), we get
	\begin{eqnarray*}
	\lefteqn{\left|\left<H_tf,\,g\right>\right|}\\
	%&\leq&\left|\sum_{k=1}^{n}<H_{ct/k^{2\beta+1+2\alpha}}\left(\Phi_kH_{t(1-\sum_{m=1}^{k}c/m^{2\beta+1+2\alpha})}f\right),\,g_k>\right|\\
	%&&+\left|\left<H_{t(1-\sum_{m=1}^{n}c/m^{2\beta+1+2\alpha})}f,\,g_n\right>\right|\\
	%&\leq& \sum_{k=1}^{n}\exp{\left\{-\frac{k}{2}\left(\frac{b^\beta d^\beta}{4ct}\right)^{\frac{1}{1+2\alpha}}\right\}}||f||_1||g||_2+\left|\left<H_{t(1-\sum_{m=1}^{n}c/m^{2\beta+1+2\alpha})}f,\,g_n\right>\right|\\
	&\leq& C\exp{\left\{-\frac{1}{2}\left(\frac{b^\beta d^\beta}{4^{1+\alpha}ct}\right)^{\frac{1}{1+2\alpha}}\right\}}||f||_1||g||_2+\left|\left<H_{t(1-cs_n)}f,\,g_n\right>\right|.
	\end{eqnarray*}
	Here $C=\max\limits_{0<t<t_0} \left(1-\exp{\left\{-\frac{1}{2}\left(\frac{b^\beta d^\beta}{4^{1+\alpha}ct}\right)^{\frac{1}{1+2\alpha}}\right\}}\right)^{-1}$. Next apply (\ref{remainingterm}) to the proper subsequence and take $\liminf$, we get (\ref{lem2}), namely
	\begin{eqnarray*}
	\left|\left<H_tf,\,g\right>\right|
	\leq C\exp{\left\{-\frac{1}{2}\left(\frac{b^\beta d^\beta}{4^{1+\alpha}ct}\right)^{\frac{1}{1+2\alpha}}\right\}}||f||_1||g||_2+\left|\left<H_{t\delta}f,\,\widetilde{g}\right>\right|.
	\end{eqnarray*}
\end{proof}
To proceed in proving the $L^\infty$\ off-diagonal upper bound, we iterate the result of Lemma \ref{lem1} to get the following lemma.
\begin{lemma}
	\label{lemma2}
	Under the hypotheses of Lemma \ref{lem1}, \begin{eqnarray}
	\left|\left<H_tf,\,g\right>\right|\leq C^2\exp{\left\{-\frac{1}{2}\left(\frac{b^\beta d^\beta}{4^{1+\alpha}ct}\right)^{\frac{1}{1+2\alpha}}\right\}}||f||_1||g||_2.\label{mainlemma}
	\end{eqnarray}
\end{lemma}
\begin{proof}
We rename $\widetilde{g}$\ by $\widetilde{g}_{\scaleto{\delta}{4pt}}$\ to indicate its co-appearance with the term $H_{\delta t}f$, and $\widetilde{V}$\ by $\widetilde{V}_\delta$\ for the same reason. Then $\rho_X(U,\widetilde{V}_{\delta})\geq (1+\nu)d$, and Lemma \ref{lem1} reads as
\begin{eqnarray*}
	\left|\left<H_tf,\,g\right>\right|\leq C\exp{\left\{-\frac{1}{2}\left(\frac{b^\beta d^\beta}{4^{1+\alpha}ct}\right)^{\frac{1}{1+2\alpha}}\right\}}||f||_1||g||_2+\left|\left<H_{\delta t}f,\,\widetilde{g}_{\scaleto{\delta}{4pt}}\right>\right|.
\end{eqnarray*}
To start with the iteration, we repeat Lemma \ref{lem1} to the pair of functions $f\in L^1(U)$, $\widetilde{g}_{\scaleto{\delta}{4pt}}\in L^2(\widetilde{V}_\delta)$. Note that here $\rho_X(U,\widetilde{V}_\delta)=(1+\nu)d$, repeating the procedure in Lemma \ref{lem1} leads to an $L^2$\ function $\widetilde{g}_{\scaleto{\delta^2}{4pt}}$\ supported in a measurable set $\widetilde{V}_{\delta^2}$\ that satisfies $\widetilde{V}_\delta\subset \widetilde{V}_{\delta^2}$, and $\rho_X(U,\widetilde{V}_{\delta^2})\geq (1+\nu^2)d$. Here we use the subscript $\delta^2$\ to indicate the co-appearance of the set and function with the term $H_{\delta^2t}f$\ in the estimate
\begin{eqnarray*}
	\left|\left<H_{\delta t}f,\,\widetilde{g}_{\scaleto{\delta}{4pt}}\right>\right|\leq C\exp{\left\{-\frac{1}{2}\left(\frac{b^\beta (\nu d)^\beta}{4^{1+\alpha}c\delta t}\right)^{\frac{1}{1+2\alpha}}\right\}}||f||_1||\widetilde{g}_{\scaleto{\delta}{4pt}}||_2+\left|\left<H_{\delta^2 t}f,\,\widetilde{g}_{\scaleto{\delta^2}{4pt}}\right>\right|.
\end{eqnarray*}
Repeating this iteration then shows that for all $N$,
\begin{eqnarray}
\left|\left<H_{t}f,\,g\right>\right|
&\leq& \sum_{n=0}^{N}C\exp{\left\{-\frac{1}{2}\left(\frac{b^\beta (\nu^n d)^\beta}{4^{1+\alpha}c\delta^n t}\right)^{\frac{1}{1+2\alpha}}\right\}}||f||_1||\widetilde{g}_{\delta^n}||_2\notag\\
&&+\left|\left<H_{\delta^{\scaleto{N+1}{3pt}} t}f,\,\widetilde{g}_{\scaleto{\delta^{N+1}}{4pt}}\right>\right|.\label{finalestimate}
\end{eqnarray}
All $\left|\left|\widetilde{g}_{\scaleto{\delta^n}{4pt}}\right|\right|_2$\ are bounded above by $||g||_2$. Recall that $\nu=1-b\sum\limits_{m=1}^{+\infty}\frac{1}{m^2}$, $\delta=1-c\sum\limits_{m=1}^{+\infty}\frac{1}{m^{2\beta+1+2\alpha}}$. By picking $b>0$\ small enough, we can make $(\nu^\beta/\delta)^{1/(1+2\alpha)}\geq 2$, this guarantees the convergence of the sum in (\ref{finalestimate}): since $(\nu^\beta/\delta)^{(n+1)/(1+2\alpha)}-(\nu^\beta/\delta)^{n/(1+2\alpha)}\geq 2$, all $\left\lfloor\left(\nu^\beta/\delta\right)^{n/(1+2\alpha)}\right\rfloor$\ are distinct integers. Hence the sum in (\ref{finalestimate}) satisfies
\begin{eqnarray*}
\lefteqn{\sum_{n=0}^{N}C\exp{\left\{-\frac{1}{2}\left(\frac{b^\beta (\nu^n d)^\beta}{4^{1+\alpha}c\delta^n t}\right)^{\frac{1}{1+2\alpha}}\right\}}||f||_1||\widetilde{g}_{\scaleto{\delta^n}{4pt}}||_2}\\
%&\leq& \sum_{n=0}^{N}C\exp{\left\{-\frac{1}{2}\left(\frac{b^\beta d^\beta}{4ct}\right)^{\frac{1}{1+2\alpha}}\cdot 2^n\right\}}||f||_1||g||_2\\
&\leq &\sum_{k=1}^{2^N}C\exp{\left\{-\frac{k}{2}\left(\frac{b^\beta d^\beta}{4^{1+\alpha}ct}\right)^{\frac{1}{1+2\alpha}}\right\}}||f||_1||g||_2.
\end{eqnarray*}
Let $N$\ tend to infinity, we get that the sum part in (\ref{finalestimate}) is bounded above by
\begin{eqnarray}
C^2\exp{\left\{-\frac{1}{2}\left(\frac{b^\beta d^\beta}{4^{1+\alpha}ct}\right)^{\frac{1}{1+2\alpha}}\right\}}||f||_1||g||_2.\label{1stterminfinalestimate}
\end{eqnarray}
%The same argument works for when $(\nu^\beta/\delta)^{1/(1+2\alpha)}>2$, since then the sequence $\{(\nu^\beta/\delta)^{n/(1+2\alpha)}\}_n$\ is strictly increasing, and two adjacent terms are far apart enough ($(\nu^\beta/\delta)^{(n+1)/(1+2\alpha)}-(\nu^\beta/\delta)^{n/(1+2\alpha)}>2$). This guarantees that each $\exp{\left\{-\frac{1}{2}\left(\frac{b^\beta d^\beta}{4ct}\right)^{\frac{1}{1+2\alpha}}\cdot (\nu^\beta/\delta)^{n/\left(1+2\alpha\right)}\right\}}$\ is bounded above by a unique integer power, namely $\exp{\left\{-\frac{1}{2}\left(\frac{b^\beta d^\beta}{Lct}\right)^{\frac{1}{1+2\alpha}}\cdot \left\lfloor(\nu^\beta/\delta)^{n/(1+2\alpha)}\right\rfloor\right\}}$.\\[0.1in]

For the second term in (\ref{finalestimate}), $\left|\left<H_{\delta^{\scaleto{N+1}{3pt}} t}f,\,\widetilde{g}_{\scaleto{\delta^{\scaleto{N+1}{3pt}}}{4pt}}\right>\right|$, we use the same weak-convergence argument as in Lemma \ref{lem1}. More precisely, let $\mathbf{V}$\ be the union of the increasing set sequence $\{\widetilde{V}_{\scaleto{\delta^n}{5pt}}\}$, i.e., 
\begin{eqnarray*}
\mathbf{V}:=\bigcup_n \widetilde{V}_{\scaleto{\delta^n}{5pt}},
\end{eqnarray*}
then all $\widetilde{g}_{\scaleto{\delta^n}{4pt}}$\ are supported in $\mathbf{V}$, and $\rho_X(U,\mathbf{V})\geq d>0$. All $||\widetilde{g}_{\scaleto{\delta^n}{4pt}}||_2\leq ||g||_2$, so there exists some weakly convergent subsequence $\left\{\widetilde{g}_{\scaleto{\delta^{n_{\scaleto{k}{3pt}}}}{4pt}}\right\}_{k}$, let $\mathbf{g}\in L^2(\mathbf{V})$\ be the weak limit. Then as $k\rightarrow +\infty$,
\begin{eqnarray*}
\left<H_{\delta^{n_{\scaleto{k}{3pt}}}t}f,\,\widetilde{g}_{\scaleto{\delta^{n_{\scaleto{k}{3pt}}}}{4pt}}\right>\rightarrow <f,\,\mathbf{g}>=0,
\end{eqnarray*}
as $f$\ and $\mathbf{g}$\ have disjoint supports ($\rho_X(U,\mathbf{V})\geq d>0$). Thus by taking $\liminf$\ in (\ref{finalestimate}) and plugging in the upper bound (\ref{1stterminfinalestimate}) for the sum part, we get (\ref{mainlemma}), i.e.,
\begin{eqnarray*}
	\left|<H_tf,\,g>\right|\leq C^2\exp{\left\{-\frac{1}{2}\left(\frac{b^\beta d^\beta}{4^{1+\alpha}ct}\right)^{\frac{1}{1+2\alpha}}\right\}}||f||_1||g||_2.
\end{eqnarray*}
\end{proof}
Finally, consider two functions $u,v\in L^1\cap L^2$\ satisfying that $\mbox{supp}\{u\}\subset U$, $\mbox{supp}\{v\}\subset V$. We bound $\left|\left<H_tu,\,v\right>\right|$\ using the $L^1$\ norms of both $u$\ and $v$. For convenience we consider $\left<H_{2t}u,\,v\right>$. To apply (\ref{mainlemma}), decompose
\begin{eqnarray*}
	\lefteqn{<H_{2t}u,\,v>=<H_tu,\,H_tv>}\\
	&=&<H_tu,\,(\Phi+\Psi)H_tv>=<H_tu,\,\Phi H_tv>+<\Psi H_tu,\,H_tv>.
\end{eqnarray*}
Here $\Phi=1_{O_1}$\ and $\Psi=1_{O_2}$\ are characteristic functions and their supports partition $X$. $O_1,O_2$\ further satisfy that $\mbox{supp}\{u\}\subset U\subset O_1$, $\mbox{supp}\{v\}\subset V\subset O_2$, and $\rho_X(U, O_2),\rho_X(V, O_1)\geq \frac{1}{2}\rho_X(U,V)$. Apply (\ref{mainlemma}) to estimate $<H_tu,\,\Phi H_tv>$\ and $<\Psi H_tu,\,H_tv>$\ separately by setting $f=u$, $g=\Phi H_tv$\ for the first term, and $f=v$, $g=\Psi H_tu$\ for the second term. By Lemma \ref{lemma2}, for small enough $b,c>0$, for $d=\frac{1}{4}\rho_X(U,V)$, there exists some $t_0>0$, such that for all $0<t<t_0$,
\begin{eqnarray*}
	\lefteqn{\left|\left<H_{2t}u,\,v\right>\right|}\\
	&\leq& \left|\left<H_tu,\,\Phi H_tv\right>\right|+\left|\left<\Psi H_tu,\,H_tv\right>\right|\\
	&\leq& C^2\exp{\left\{-\frac{1}{2}\left(\frac{b^\beta d^\beta}{4^{1+\alpha}ct}\right)^{\frac{1}{1+2\alpha}}\right\}}\left(||u||_1||\Phi H_tv||_2+||v||_1||\Psi H_tu||_2\right)\\
	&\leq& C^2\exp{\left\{-\frac{1}{2}\left(\frac{b^\beta d^\beta}{4^{1+\alpha}ct}\right)^{\frac{1}{1+2\alpha}}\right\}}\cdot 2\left|\left|u\right|\right|_1\left|\left|v\right|\right|_1e^{M(t)}.
\end{eqnarray*} 
By the assumption on $M(t)$\ (\ref{M(t)}),
\begin{eqnarray*}
C_0:=\sup_{0<t<t_0}C^2\exp{\left\{-\frac{1}{4}\left(\frac{b^\beta d^\beta}{4^{1+\alpha}ct}\right)^{\frac{1}{1+2\alpha}}\right\}}\cdot 2e^{M(t)}<+\infty.
\end{eqnarray*}
This completes the proof of Theorem \ref{appendixprop}.

As in the $L^2$\ case (cf. \cite{L2}), we can similarly generalize the above result to time derivatives of $<H_tu,\,v>$.

\bibliographystyle{plain}
\bibliography{hyplinf}

\end{document}